\documentclass[a4paper]{amsart}
\usepackage[english]{babel}
\usepackage{amssymb}
\usepackage{amsmath}
\usepackage{amsthm}
\usepackage{graphicx}
\usepackage{fancyhdr}
\usepackage{bbm}
\usepackage{multirow}
\newfont{\bssten}{cmssbx10}
\newfont{\bssnine}{cmssbx10 scaled 900}
\newfont{\bssdoz}{cmssbx10 scaled 1200}

\usepackage{latexsym,amssymb,amsthm,amsxtra}
\usepackage{amsmath}
\usepackage{stmaryrd}
\usepackage{mathrsfs}
\usepackage{tikz}
\usepackage{pgf}
{\bf}{\it}
{\bf}{\it}
\newtheorem{theorem}{Theorem}
\newtheorem{definition}{Definition}
\newtheorem{lemma}{Lemma}
\newtheorem{remark}{Remark}
\newtheorem{proposition}{Proposition}

\newtheorem{ex}{Example}

\def\/{\, | \,}

\def\ind{{\mathchoice {\rm 1\mskip-4mu l} {\rm 1\mskip-4mu l}
		{\rm 1\mskip-4.5mu l} {\rm 1\mskip-5mu l}}}
\def\N{{\mathbb N}}

\def\P{{\mathbb P}}

\newcommand{\B}{{\mathcal B}}
\newcommand{\A}{{\mathcal A}}

\def\C{{\mathcal C}}
\def\I{{\mathbb I}}

\def\esp#1{{\mathbb E}\left[#1\right]}
\def\cop#1{\underline{#1}}

\newcommand{\pae}[1]{\mbox{$\lfloor \kern-1pt #1 \kern-1pt \rfloor$}}
\newcommand{\paep}[1]{\mbox{$\lceil \kern-1pt #1 \kern-1pt \rceil$}}

\newcommand\ccc{\circledcirc}

\def\E{{\mathcal E}}

\def\N{{\mathbb N}}
\def\R{{\mathbb R}}
\def\maV{{\mathcal V}}
\def\maZ{{\mathcal Z}}
\def\maP{{\mathcal P}}

\def\maE{{\mathcal E}}

\def\mbX{{\mathbb X}}
\def\mbW{{\mathbb W}}
\def\bw{\mathbf x}

\def\v{{\--}}
\def\pv{{\not\!\!\--}}

\def\bw{\mathbf w}
\newcommand\td[1]{\overline{#1}}
\newcommand\pr[1]{{\mathbb P}\left[#1\right]}

\newcommand\suite[1]{\left\{#1\right\}_{n\in\N}}

\newcommand\gre{\textbf{e}}

\newcommand\maM{{\mathcal M}}

\newcommand\maA{{\mathcal A}}

\newcommand\maI{{\mathcal I}}

\def\isdef{\triangleq}
\def\fcfm{\textsc{fcfm}}

\DeclareMathOperator{\card}{Card}

\def\V{\mathcal V}
\def\A{\mathcal A}
\def\B{\mathcal B}
\def\C{\mathcal C}
\def\bV{\mathbf V}
\def\oV{\overline{\V}}
\def\w{\mathbf w}
\def\E{\mathcal E}
\def\S{\mathcal V_1}
\def\N{\mathbb N}
\def\P{\mathbb P}
\def\isdef{\triangleq}
\def\fcfm{\textsc{fcfm}}

\def\ll{[\![}
\def\rr{]\!]}
\def\mudeg{\mu_{\tiny{\mbox{deg}}}}
\def\out{{\mathcal S}}

\usepackage[utf8]{inputenc}

\begin{document}
	
	\title[Stochastic matching model with self-loops]{A general stochastic matching model on multigraphs}
	\author[Begeot, Marcovici, Moyal, Rahme]{Jocelyn Begeot, Irène Marcovici, Pascal Moyal, and Youssef Rahme}
	
	\begin{abstract} 
		We extend the general stochastic matching model on graphs introduced in \cite{MaiMoy16}, to matching models on multigraphs, that is, graphs with self-loops. 
		The evolution of the model can be described by a discrete time Markov chain whose positive recurrence is investigated. 
		Necessary and sufficient stability conditions are provided, together with the explicit form of the stationary probability in the case where the matching policy is `First Come, First Matched'.   
	\end{abstract}

\maketitle
\section{introduction}
\label{introduction}
Over the past decade, an increasing interest has been dedicated to stochastic systems in which incoming elements are matched according to specified compatibility rules. 
This is, first, a natural representation of service systems in which customers and servers are of different classes, and where designated classes of servers can serve 
designated classes of customers. For this general class of queueing models, termed skill-based queueing systems, it is then natural to investigate the conditions for the existence of a stationary state, 
and under these conditions, to design and control the model at best, for given performance metrics (end-to-end delay, matching rates, fairness, etc.). Such models are classical queueing systems, in the sense that 
there is a dissymmetry between customers and servers: customers come and depart the system, whereas servers are part of the `hardware', remain in the system, and switch to the service of another customer when they have completed one (with possible vacation times in-between services). 

In \cite{CKW09} (see also \cite{AdWe}), a variant of such skill-based systems was introduced, which are now commonly referred to as `Bipartite matching models' (BM): couples customer/server enter the system at each time point, and customers and servers play symmetric roles: exactly like customers, servers come and go into the system. Upon arrival, they wait for a compatible customer, and as soon as they find one, leave the system together with it. These settings are suitable to various fields of applications, among which, blood banks, organ transplants, housing allocation, job search, dating websites, and so on.  In both references, compatible customers and servers are matched according to the FCFS (`First Come, First Served') service discipline. 
In \cite{ABMW17}, a subtle dynamic reversibility property is shown, entailing that the stationary state of such systems under FCFS, can be obtained in a product form. Moreover, a sub-additivity property is proved, allowing (under stability conditions) the construction of a unique stationary bi-infinite matching of the customers and servers, by a coupling-from-the-past (CFTP) technique. 
Interestingly, the product-form of the stationary state can then be adapted to 
various skill-based queueing models as well, and in particular those applying (various declinations of) the so-called FCFM-ALIS (Assign the Longest Idling Server) service discipline - see e.g. \cite{AW14}, and various extensions of BM models in \cite{AKRW18,BC15,BC17}. 
 In \cite{BGM13}, the settings of  \cite{CKW09,AdWe} are generalized to more general service disciplines (termed `matching policies' in this context), and necessary and sufficient conditions for the stability of the system are introduced, which are functions of the compatibility graph and of the matching policy. Moreover, as opposed to the previously cited references, the results in \cite{BGM13} do not assume the independence between the types of the entering customer and the entering server. The system is then called Extended Bipartite Matching Model (EBM, for short), and suits  applications in which independence between the classes of the 
 customers and servers entering simultaneously cannot be assumed. In \cite{MBM18}, a CFTP result is obtained, showing the existence of a unique bi-infinite matching in various cases for EBM models, and for a broader class of matching policies than FCFS, thereby generalizing the results of \cite{ABMW17}. 
 
For the purpose of modeling concrete systems, the need then arose to extend these different models. Indeed, in many applications the assumption of pairwise arrivals may appear somewhat artificial, and it is more realistic to assume that arrivals are simple. In addition, all the aforementioned references assume that the compatibility 
 graph is bipartite, namely there are easily identifiable classes of {\em servers} and classes of {\em customers}, whatever these mean: donors/receivers, houses/applicants, jobs/applicants, and so on. However, in many cases the context requires that the compatibility graph take a general (i.e., not necessarily bipartite) form. For instance, in dating websites, it is a priori not possible to split items into two sets of classes (customers and servers) 
 with no possible matches within those sets. Similarly, in kidney exchange programs, intra-incompatible couples donor/receiver enter the system, looking for a compatible couple to perform a `crossed' transplant. 
 Then, it is convenient to represent couples donor/receiver as  {\em single} items, and compatibility between couples 
 means that a kidney exchange can be performed between the two couples (the donor of the first couple can give to the receiver of the second, and the donor of the second can give to the receiver of the first). In particular, if one consider blood types as a primary compatibility criterion, the 
 compatibility graph between couples is naturally non-bipartite. 
 Motivated by these observations, a variant model was introduced in \cite{MaiMoy16}, in which items arrive one by one and the compatibility graph is general, i.e., not necessarily bipartite: specifically, in this so-called {\em General stochastic Matching model} (GM for short), items enter one by one in discrete time in a buffer, and belong to determinate classes in a finite set $\maV$. Upon each arrival, the class of the incoming item is drawn independently of everything else, from a distribution $\mu$ having full support $\maV$. A connected graph $G$ whose set of nodes is precisely $\maV$, determines the compatibility among classes. 
 Then, an incoming item is either immediately matched, if there is a compatible item in line, or else stored in a buffer. It is the role of the matching policy $\Phi$ to determine the match of the incoming item in case of a multiple choice. Then, the two matched items immediately leave the system forever. The {\em stability region} of the model, given $G$ and $\Phi$, is then defined as the set of measures $\mu$ such that the model is 
 	positive recurrent.  A necessary condition for the stability of GM models is provided in \cite{MaiMoy16} (see (\ref{eq:Ncond}) below). In particular, the latter condition is empty if and only the compatibility graph is bipartite (which partly justifies why items enter by pairs in BM and EBM models - otherwise the model could not be stabilizable). On another hand, it is proven in \cite{MaiMoy16} that the 
 matching policy `Match the Longest' has a maximal stability region, that is, the latter necessary condition is also sufficient (we then say that the latter policy is {\em maximal}). However, \cite{MoyPer17} shows  that in fact, aside for a particular class of graphs, random policies are never maximal, and that there always exists a strict priority policy that isn't maximal either. Then, by adapting the dynamic reversibility argument of \cite{ABMW17} to the GM models, \cite{MBM17} shows that the matching policy First Come, First Matched (FCFM) is also maximal, and derives the stationary probability in a product form. More recently, following the work of \cite{NS17}, 
 matching policies of the broader {\em Max-Weight} type (including `Match the Longest') are shown to be maximal, and drift inequalities allow to bound the speed of convergence to the equilibrium, and the first two moments of 
 the stationary state. Variants of the GM model to the case of (i) hypergraphical structures (i.e. matching items by groups of two or more) and (ii) graphical systems with reneging 
 are investigated, respectively in \cite{GW14,NS17,RM19} and \cite{JMRS20} (see also \cite{BDPS11}). 
 
Motivated again by concrete applications, in the present article we present a further extension of the GM model. Indeed, in various contexts, among which dating websites and peer-to-peer interfaces, it is natural to assume that items {\em of the same class} can be matched together. 
 Hence, the need to generalize the previous line of research to the case where the matching architecture is a {\em multigraph} (a graph admitting {\em self-loops}, that is, edges connecting nodes to themselves), rather than just a graph. This generalization is the core subject of the present paper. We show how several stability results of \cite{MaiMoy16,MBM17,JMRS20} can be generalized to the case of a multigraphical matching structure. As is easily seen, the buffer of a matching model on a multigraph is hybrid by essence: nodes admitting self-loops (if any)  admit at most one item in line, whereas nodes with no self-loops (if any) have unbounded queues. A matching model on a multigraph typically has a larger stability region than the corresponding model on a graph on which all self-loops are erased (the {\em maximal subgraph} of the latter -  see Definition \ref{def:restricted}), but the interplay between self-looped nodes and their non-self-looped neighbors needs to be clearly understood: intuitively, the arrival flows to self-looped nodes appear as auxiliary flows helping their neighboring non-self-looped nodes to stabilize their own queues - provided that the arrivals to self-looped nodes don't match too often with one another. 
 On another hand, in the extreme case where all nodes are self-looped, the model is finite and 
 has a flavor of statistical-physical system: it is an irreducible Markov chain on $\{0,1\}^{|\maV|}$ with local interactions - see Example \ref{ExampleSquare} below. 
 {\em En route}, by showing results for stochastic matching models on multigraphs, we show various results that have their own inner interest for GM models on graphs - see in particular Propositions \ref{prop:ncond} and \ref{prop:extppartite} and Sub-section \ref{subsec:iid} below.

This paper is organized as follows: we start by some preliminary in Section \ref{sec:prelim}, and in particular by introducing the main definitions and properties of multigraphs. 
In Section \ref{sec:model}, we formally introduce the present model. In Section \ref{sec:results}, we present our main results for GM models on multigraphs, among which, the maximality and the explicit product form of the stationary probability for the FCFM policy, and the maximality of Max-Weight policies. To illustrate these results, several examples are presented in Section \ref{sec:examples}. The proofs of our main results are then presented in Sections \ref{sec:ncond}, \ref{sec:FCFM} and \ref{sec:otherproofs}.

\section{Preliminary}
\label{sec:prelim}
\subsection*{General notation} 
We denote by $\R$ the set of real numbers, by $\N$ the set of non-negative integers and by $\N_+$ the subset of positive integers. For any $p,q \in \N_+$, we denote by $\llbracket p,q \rrbracket$ the integer interval $[p,q] \cap \N_+$. 

We let $\mathfrak{S}_n$ denote the symmetric group on the set $\llbracket 1,n \rrbracket$, i.e. the set of permutations of $\llbracket 1,n \rrbracket$.

For any finite set $A$, we denote by $|A|$ the cardinality of $A$. The set $A$ is often implicitly ordered, and identified with $\llbracket 1,|A| \rrbracket$. 
The set of probability measures having full support on $A$ is denoted by $\mathscr M(A)$. 

\subsection*{Words} For $k\in\N_+$ and $k$ finite sets $A_1,\dots,A_k$, we identify the cartesian product $A_1 \times A_2 \times\dots\times A_k$ with the set, 
denoted by $A_1A_2\dots A_k$, of \emph{words} of length $k$ whose $i$-th letter is an element of $A_i$, for all $i\in\llbracket 1,k \rrbracket$. In particular,  $A^k$ is identified with the set of words of length $k$ over the alphabet $A$. 
We then  
denote by $A^*\isdef \cup_{k\in\N}A^k$ the set of {finite words} over the alphabet $A$. We denote the \emph{length} of $w\in A^*$ by $|w|$, so that if $w\in A^k, |w|=k$.  

For $w \in A^*$ and $B\subset A$, we introduce the notation $|w|_B\isdef\card\{i\in\{1,\ldots,|w|\} :  w_i\in B\}$. For a letter $a\in A$, we denote simply by $|w|_a\isdef |w|_{\{a\}}$ the number of occurrences of the letter $a$ in the word $w$. 

The concatenation of $k$ words $w^1,w^2,\dots,w^k$ of $A^*$, that is, the word $w$ in which appear successively from left to right, the words $w^1, w^2, \dots, w^k$, is denoted by $w=w^1w^2\dots w^k$. In particular, for a word $w$ and a letter $i$, the word 
$wi$ denotes the concatenation of the word $w$ with the single-letter word $i$, in that order. 
Any word $w\in A^*$ of length $|w|=q$ is written $w=w_1w_2\dots w_q$, and for any $i\in\llbracket 1,n \rrbracket,$ we denote by $w_{[i]}$ the word of length $|w|-1$ obtained from $w$ by deleting its $i$-th letter $w_i$. The empty word (i.e. the unique word of $A^*$ of length 0) is denoted 
$\varepsilon$. For any word $w=w_1w_2...w_q$, the {\em prefix} of $w$ of length $k\le q$ is the word $w'=w_1w_2...w_k$.  

For any $w\in A^*$, we let $[w]\isdef(|w|_a)_{a\in A}\in \N^{|A|}$ be the {\em commutative image} of $w$. For any integer $q$, the vectors of $\N^q$ are denoted as 
$\bw\isdef(w(1),\dots,w(q))$. We let, for any $i\in \llbracket 1,q \rrbracket$, $\gre_i$ be the $i$-th vector of the canonical basis of $\R^q$.

\subsection*{Multigraphs}
Hereafter, a {\em multigraph} is given by a couple $G=(\maV,\maE)$, where $\maV$ is the (finite) set of nodes and $\maE\subset \maV\times \maV$ is the set of edges. All graphs considered hereafter are undirected, 
that is, $(u,v)\in \maE$ $\implies$ $(v,u)\in \maE$, for all $u,v \in \maV$.  We write $u \v v$ or $v \v u$ for $(u,v) \in \maE$, 
and $u \pv v$ (or $v \pv u$) else. Elements of the form $(v,v) \in \maE$, for $v\in\maV$, are called {\em self-loops}. 
For any multigraph $G=(\maV,\maE)$ and any $U\subset\maV$, we denote 
\[
\maE(U) \isdef \{ v \in \maV  :  \exists u \in U, \ u
\-- v\}
\]
the neighborhood of $U$, and for $u\in \maV$, we write for short $\maE(u)\isdef\maE(\{u\})$. The set $\maV$ can then be partitioned in $\maV =\maV_1 \cup \maV_2$, where $\maV_1\isdef \{u\in\maV:u \v u \}$ and $\maV_2\isdef \{u\in\maV: u\pv u\}$, i.e., $\maV_1$ contains all nodes from which a self-loop emanates, if any, and $\maV_2$ is the complement set of $\maV_1$ in $\maV$.
A multigraph having no self-loop, that is, a couple $G=(\maV,\maE)$ such that $\maV_1=\emptyset$, is simply a {\em graph}. 
Observe that, with respect to the classical notion of multigraphs, we assume hereafter that all edges are simple. 
For any node $i\in\maV$, we denote by 
\[\mbox{deg}(i) = |\maE(i)|,\]
 the {\em degree} of node $i$, that is, the number of neighbors of $i$ (possibly including $i$ itself if $i\in\maV_1$). 
Notice that, by defining $\maE$ as a set of (ordered) couples as we did above, for any $i,j$ such that $i\v j$, $(i,j)$ appears in $\maE$ together with $(j,i)$. In particular we get that 
\begin{equation}
\label{eq:nummudeg}
\sum_{i\in\maV}\mbox{deg}(i) = |\maE|.
\end{equation}

For any multigraph $G=(\maV,\maE)$ and any $U\subset \maV$, the {\em
	subgraph induced by} $U$ in $G$ is the multigraph $(U,\maE\cap (U\times U))$. 
%
An {\em independent set} of $G$ is a non-empty subset $\maI \subset \maV$ which does not include any pair of neighbors, {\em i.e.} : $\bigl(\forall (i,j) \in \maI^2, \ i \pv j\bigr)$. 
We let $\I(G)$ be the set of independent sets of $G$. 
Then, observe that $\forall \mathcal{I}\in \mathbb{I}(G), \; \mathcal{I}\cap\mathcal{V}_1=\emptyset$, i.e. $\mathcal{I}\subset\mathcal{V}_2$.
An independent set is said maximal, if it is not strictly included in another independent set.

Let us now recall and introduce the following definitions:

\begin{definition}
A graph $G=(\maV,\maE)$ is said to be {complete $q$-partite} (or blow-up graph of order $q$) if $\maV$ can be partitioned into $q$ maximal independent sets $\maI_1,\dots,\maI_q$, such that 
for any $k,\ell \in \llbracket 1,q \rrbracket$, $k\ne \ell$ and any $i_k \in \maI_k$ and $i_\ell \in \maI_\ell$, we have that $i_k \v i_\ell$. \end{definition}

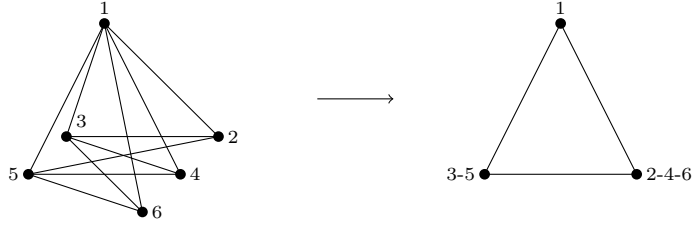
\begin{figure}[htb]
	\begin{center}
		\begin{tikzpicture}
		\fill (1,2) circle (2pt)node[above]{\scriptsize{1}};
		\fill (0,0) circle (2pt)node[left]{\scriptsize{5}};
		\fill (0.5,0.5) circle (2pt)node[above]{\scriptsize{$\;\;\;\;\;3$}};
		\fill (2,0) circle (2pt)node[right]{\scriptsize{4}};
		\fill (2.5,0.5) circle (2pt)node[right]{\scriptsize{2}};
		\fill (1.5,-0.5) circle (2pt)node[right]{\scriptsize{6}};
		\draw[-] (2.5,0.5) -- (1,2);
		\draw[-] (0.5,0.5) -- (1,2);
		\draw[-] (2,0) -- (1,2);
		\draw[-] (0,0) -- (1,2);
		\draw[-] (0,0) -- (2,0);
		\draw[-] (0,0) -- (2.5,0.5);
		\draw[-] (0.5,0.5) -- (2,0);
		\draw[-] (0.5,0.5) -- (2.5,0.5);
		\draw[-] (1.5,-0.5) -- (0.5,0.5);
		\draw[-] (1.5,-0.5) -- (0,0);
		\draw[-] (1.5,-0.5) -- (1,2);
		\draw[->] (3.8,1) -- (4.8,1);
			\fill (7,2) circle (2pt)node[above]{\scriptsize{1}};
		\fill (6,0) circle (2pt)node[left]{\scriptsize{3-5}};
		\fill (8,0) circle (2pt)node[right]{\scriptsize{2-4-6}};
			\draw[-] (7,2) -- (8,0);
		\draw[-] (7,2) -- (6,0);
		\draw[-] (6,0) -- (8,0);
		\end{tikzpicture}
		\caption{Complete $3$-partite subgraph (left), complete graph of order 3 (right).}
		\label{fig:Graph.q-partiteAndComplementaire}
	\end{center}
\end{figure}

In other words, a {complete $q$-partite} graph is a blow-up of the complete graph of {size} $q$, in the sense that all nodes are duplicated into one, or several distinct and unconnected nodes having the same neighbors as the original node. See an example on Figure \ref{fig:Graph.q-partiteAndComplementaire}. 

\begin{definition}
	\label{def:restricted}
	Let $G=(\maV,\maE)$ be a multigraph. The {\em maximal subgraph} of $G$ is the graph $\check G=(\maV,\check{\maE})$ 
	obtained by deleting all self-loops in $G$, that is 
	\begin{equation}
	\label{eq:defcheckE}
	\check{\maE}\isdef\maE\setminus\left\lbrace(i,i):i\in\maV_1\right\rbrace.
	\end{equation}
	See an example on Figure \ref{fig:GgraphGLmultigZAndGTilde}. 
\end{definition}

\begin{definition}
	\label{def:extended}
	Let $G=(\maV_1\cup\maV_2, \maE)$ be a multigraph. 
	The {\em minimal blow-up graph} of $G$ is the graph $\hat{G}=(\hat{\maV},\hat{\maE})$ defined as follows:
\begin{equation}
\label{eq:defGtilde}
\hat{\maV}\isdef\maV\cup\cop{\maV_1}\quad \textrm{and}\quad\hat{\maE}\isdef\check{\maE} \cup\underline{\maE_1},
\end{equation}
where $\cop{\maV_1}$ is {a} disjoint copy of $\maV_1$, $\check{\maE}$ is defined by (\ref{eq:defcheckE}) and $$\underline{\maE_1}\isdef\{(\underline{i},j):(i,j)\in\maE,\,i\in\maV_1,\,j\in\maV\}.$$
In other words, $ \hat G$ is obtained from $G$ by duplicating each node having a self-loop by two nodes having the same neighborhood and replacing each self-loop by an edge between 
the node and its copy. See an example {in} Figure \ref{fig:GgraphGLmultigZAndGTilde}. For any set $\maA \subset \maV_1$, we denote by $\cop{\maA}$ the set of all copies of elements of $\maA$, that is 
\[\cop{\maA} \isdef \left\{\underline i : i\in \maA\right\}.\]
For any element $j\in\cop{\mathcal A}$, we will also denote by $\cop{j}$ the unique element $i\in\A$ such that $j=\cop{i}$. 
The maximal subgraph $\check G$ of $G$ is then called {\em reduced graph} of $\hat G$. 
\end{definition}

\begin{figure}[htb]
	\begin{center}
		\begin{tikzpicture}
		\fill (0,2) circle (2pt)node[above]{\scriptsize{1}};
		\fill (0,1) circle (2pt)node[right]{\scriptsize{2}};
		\fill (-1,0) circle (2pt)node[left]{\scriptsize{3}};
		\fill (1,0) circle (2pt)node[right]{\scriptsize{4}};
		\draw[-] (0,1) -- (1,0);
		\draw[-] (0,1) -- (-1,0);
		\draw[-] (0,1) -- (0,2);
		\draw[-] (-1,0) -- (1,0);
		\draw[<-] (1.8,1) -- (2.2,1);
		\fill (4,2) circle (2pt)node[above]{\scriptsize{1}};
		\fill (4,1) circle (2pt)node[right]{\scriptsize{2}};
		\fill (3,0) circle (2pt)node[left]{\scriptsize{3}};
		\fill (5,0) circle (2pt)node[right]{\scriptsize{4}};
		\draw[-] (4,1) -- (5,0);
		\draw[-] (4,1) -- (3,0);
		\draw[-] (4,1) -- (4,2);
		\draw[-] (3,0) -- (5,0);
		\draw[->] (5.8,1) -- (6.2,1);
		\draw[-] (8,1) -- (9,0);
		\draw[-] (8,1) -- (7,0);
		\draw[-] (8,1) -- (8,2);
		\draw[-] (7,0) -- (9,0);
		\fill (8,2) circle (2pt)node[above]{\scriptsize{1}};
		\fill (8,1) circle (2pt)node[left]{\scriptsize{2}};
		\fill (7,0) circle (2pt)node[left]{\scriptsize{3}};
		\fill (9,0) circle (2pt)node[right]{\scriptsize{4}};
		\fill (10,1) circle (2pt)node[right]{\scriptsize{$\cop{2}$}};
		\draw[-] (10,1) -- (9,0);
		\draw[-] (10,1) -- (7,0);
		\draw[-] (10,1) -- (8,2);
		\draw[-] (10,1) -- (8,1);
		\draw[thick,-] (4,1) to [out=110,in=200,distance=10mm] (4,1);
		\end{tikzpicture}
		\caption{A multigraph $G$ (middle), its maximal subgraph $\check G$ (left) and its minimal blow-up graph $\hat{G}$ (right).}
		\label{fig:GgraphGLmultigZAndGTilde}
	\end{center}
\end{figure}
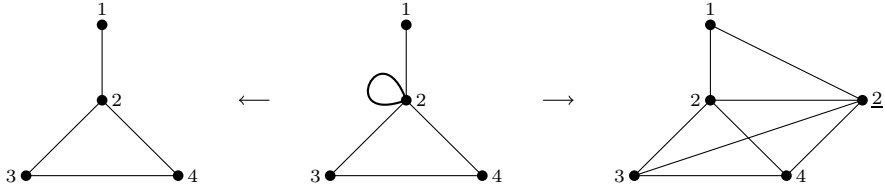

Throughout this paper, all considered multigraphs are connected, that is, for any $u,v \in \maV$, there exists a subset 
$\{v_0\isdef u, v_1, v_2,\dots, v_p\isdef v\}\subset \maV$ such that $v_i \v v_{i+1}$, for any $i \in\llbracket 0,p-1 \rrbracket$. 



\section{The model}
\label{sec:model}
All the random variables (r.v., for short) hereafter are defined on a common probability space $(\Omega,\mathcal{F},\mathbb{P}).$

\subsection{Stochastic matching model on a multigraph}
\label{sec:introd}
Generalizing the model introduced in \cite{MaiMoy16} to multigraphs (rather than only graphs), a 
{\em matching model}  is formally specified by the triple $(G,\Phi,\mu)$, where 
\begin{enumerate}
	\item $G=(\maV,\maE)$ is a connected multigraph {having at least two nodes}; 
	\item $\Phi$ is the {\em matching policy}, to be properly defined later;
	\item $\mu$ is an element of $\mathscr M(\maV)$, i.e. a probability measure having full support $\maV$. 
\end{enumerate}
\medskip
%
%

The dynamics of the matching model associated to $(G,\Phi,\mu)$  is then similar to 
that of stochastic matching models on general graphs, as in \cite{MaiMoy16}: 
start with an empty ``buffer'' and, for any $n\in\N_+$, draw an 
element $V_n$ of $\maV$ from the probability measure $\mu$, independently of 
$\sigma\left(\left\{V_1,V_2,\dots,V_{n-1}\right\}\right)$, and apply the following rules:
\begin{itemize}
\item if there is no element of class $i$ in the buffer such that $V_n \--i$, then add an item of class $V_n$ to the buffer;
\item otherwise, do not add $V_n$ and remove from the buffer an element of class $i$ such that $V_n\-- i$ 
(we say that $V_n$ and $i$ are {\em matched} together). If several elements $i$ of the buffer are such that $V_n\-- i$, the one to be
removed depends on a {\em matching policy} to be specified for the considered model. 
\end{itemize}
\subsection{State spaces}
\label{subsec:state}
We reproduce here the state description of the model introduced in \cite{MaiMoy16} for the stochastic model on general graphs, and then 
\cite{MBM18} for the same model under {\sc fcfm}. 
Fix a connected multigraph $G=(\maV,\maE)$, in the sense specified above, until the end of this section. 
Fix an integer $n_0 \ge 1$, a realization $v_1,\dots,v_{n_0}$ of $V_1,\dots,V_{n_0}$, and define the word $z\isdef v_1\dots v_{n_0} \in \maV^*$.  
Then, for any matching policy $\Phi$, there exists a unique {\em matching} of the word $z$, that is, a graph on the set of nodes 
$\left\{v_1,\dots,v_{n_0}\right\}$, whose edges represent the matches performed in the system until time $n_0$, if the successive arrivals are given by $z$.   
This matching is denoted by $M^\Phi(z)$. 
The state of the system is then defined as the word $W^\Phi(z)\in \maV^*$, whose letters are the classes of the unmatched items at time $n_0$, 
i.e. the isolated vertices in the matching $M^{\Phi}(z)$, in their order of arrivals. The word $W^\Phi(z)$ is called {\em queue detail} at time $n_0$. 
Then, any admissible queue detail belongs to the set 
\begin{equation}
\mathbb W \isdef \Bigl\{ w\in \maV^*\; : \; \forall  i\neq j \; \text{s.t.} \; (i,j) \in \maE, \; |w|_i|w|_j=0 \; \text{and} \; \forall i\in\mathcal{V}_1,\,|w|_i\leq 1 \Bigr\}.\label{eq:defmbW}
     \end{equation} 
As will be seen below, depending on the service discipline $\Phi$, we can also restrict the available information on the state of the system at time $n_0$, to a vector only keeping track of 
the number of items of the various classes remaining unmatched at $n_0$, that is, of the number of occurrences of the various letters of the alphabet $\maV$ in the word $W^\Phi(z)$.    
This restricted state thus equals the commutative image of $W^{\Phi}(z)$ and is called {\em class detail} of the system. It takes values in the set 
\begin{align}
\mathbb X &\isdef \Bigl\{x \in \N^{|\maV|}\,:\,\forall i\neq j \mbox{ s.t. }(i,j)\in\maE,\; x(i)x(j)=0\mbox{ and }\forall i\in\maV_1,\, x(i)\leq 1\Bigl\}\nonumber\\
&=\Bigl\{\left[w\right]\;:\,w \in \mathbb W\Bigl\}.\nonumber
\end{align} 

\subsection{Matching policies}
\label{subsec:pol}
We now present and  formally define the set of usual matching policies that can be considered, 
\begin{definition}
A matching policy $\Phi$ is said {\em admissible} if the choice of the match of an incoming item depends 
{solely} on the queue detail upon the arrival (and possibly on an independent uniform draw, in case of a tie). 
\end{definition} 
An admissible matching policy can be formally characterized by an action $\odot_{\Phi}$ of $\maV$ on $\mathbb W$, defined as follows: 
if $w$ is the queue detail at a given time and the input is augmented by the arrival of $v \in \maV$ at that time, then the new queue detail $w'$ and $w$ {satisfy} the relation
\begin{equation}
\label{eq:defodot}
w'= w\odot_{\Phi} v.
\end{equation}
Notice that the action $\odot_\Phi$ is possibly random. See below. 

\subsubsection{Matching policies that depend on the arrival times}  
In 'First Come, First Matched' ({\sc fcfm}), the oldest item in line is chosen, so the map  $\odot_{\textsc{fcfm}}$ is 
given, for all $w \in \mathbb W$ and all $v\in\mathcal{V}$, by 
$$
w \odot_{\textsc{fcfm}} v \isdef
\left \{
\begin{array}{ll}
wv & \textrm{if } \; |w|_{\maE(v)} = 0,\\
w_{\left [\Phi(w,v)\right]} & \textrm{else, where }\Phi(w,v) = \min \{k\in[\![1;|w|]\!] : w_k\in\mathcal{E}(v)\}.
\end{array}
\right.
$$
In `Last Come, First Matched' ({\sc lcfm}), the updating map $\odot_{\textsc{lcfm}}$ is analog to $\odot_{\textsc{fcfm}}$, for 
$\Phi(w,v) = \max \{k\in[\![1;|w|]\!] : w_k\in\mathcal{E}(v)\}.$ 

\subsubsection{Class-admissible matching policies}
A matching policy $\Phi$ is said to be {\em class-admissible} if it can be implemented upon the sole knowledge of 
the class-detail of the system. Let us define, for any $x\in \mathbb X$ and any $v \in \maV$,
\begin{equation*}
\maP(x,v) \isdef\Bigl\{j\in \maE(v):x\left(j\right) > 0\Bigl\},\label{eq:setP2}
\end{equation*}
the set of classes of available compatible items with the entering class $v$-item, if the class-detail of the system is given by $x$. 
Then, a class-admissible policy $\Phi$ is fully characterized by a (possibly random) mapping $p_\Phi$, such that $p_\Phi(x,v)$ denotes the class of the match chosen 
by the entering $v$-item under $\Phi$, in a system of class-detail $x$, such that $\mathcal P(x,v)$ is non-empty. 
Then, the arrival of $v$ entails the following action on
the class-detail:
\begin{equation}
\label{eq:defccc}
x \ccc_{\Phi} v \isdef \left \{
\begin{array}{ll}
x+\gre_v &\mbox{ if }\mathcal P(x,v)=\emptyset,\\
x-\gre_{p_\Phi(x,v)}&\mbox{ else}. 
\end{array}
\right .
\end{equation}
\begin{remark}
\label{rem:equiv}
\rm
As is easily seen, to any class-admissible policy $\Phi$ corresponds an admissible policy, if one makes precise the rule of choice of match for the 
incoming items {\em within} the class that is chosen by $\Phi$, in the case where more than one item of that class is present in the system. 
In this paper, we always {assume} that within classes, the item chosen is always the {\em oldest} in line, i.e. we always apply the 
\textsc{fcfm} policy {\em within classes}.  
Under this convention, any class-admissible policy $\Phi$ is admissible, that is, the mapping $\ccc_\Phi$ from 
$\mathbb X\times \maV$ to $\mathbb X$ can be detailed into a map $\odot_{\Phi}$ from 
$\mathbb W \times \maV$ to $\mathbb W$, as in (\ref{eq:defodot}), such that for any queue detail $w$ and any $v$,
\[\left[w\odot_\Phi v\right] = \left[w\right]\ccc_\Phi v.\]    
\end{remark}

\paragraph{{\em Class-admissible random policies.}} 
In a random policy, the only information that is needed to determine the choice of the match of an incoming item is whether its various compatible classes have an empty queue or not.  
Specifically, the order of preference of each incoming item is drawn upon the arrival following a given probability distribution; then, the considered item investigates its compatible classes in that order, 
until it finds one having a non-empty buffer, if any. The incoming item is then matched with an item of the latter class. In other words, upon each arrival of an item of class $v$, a permutation 
$\sigma=\left(\sigma(1),\dots,\sigma\left(|\maE(v)|\right)\right)$ is drawn from a given probability distribution on $\mathfrak S_{\maE(v)}$, and we set   
\begin{equation}
p_{\Phi}(x,v)\isdef\sigma(k),\mbox{ where }k=\min \Bigl\{i \in
\maE(v):\sigma(i)\in \maP(x,v)\Bigl\}\label{eq:pphirandom}.
\end{equation}
In particular, the uniform policy {\sc u} corresponds to the uniform distribution on $\mathfrak{S}_{\maE(v)}$.

\medskip

\paragraph{{\em Priority policies}} 
In a priority policy, for any $v\in \maV$, the order of preference of $v$ in $\maE(v)$ is fixed beforehand. In other words, upon the arrival of an item of 
class $v$, the permutation $\sigma$ of $\mathfrak S_{\maE(v)}$ is deterministic and corresponds to the order of preference of class $v$-item in $\maE(v)$. 
This is thus a particular case of random policy. 

\medskip

\paragraph{{\em Max-Weight policies}}
The Max-Weight (\textsc{mw}) policies are an important class of class-admissible policies, in which matches are based upon the queue length and 
a fixed reward that is associated to each match. 
Formally, for any {$(i,j)\in \maE$}, we let $w_{i,j}$ be the reward associated to the match of an $i$-item with a $j$-item, and fix a real parameter $\beta$. 

Then, in a system of class-detail $x$, the match of the incoming $v$-item is given by 
\begin{equation*}
p_{\Phi}(x,v)\isdef\mbox{argmax}\left\{\beta x(j) + w_{v,j}\,:\,j\in \maP(x,v)\right\},
\end{equation*}
where ties are broken uniformly at random whenever the above is non-unique. In other words, $j$ maximizes a linear combination of the queue-size 
and the rewards. Several particular cases are to be mentioned:
	\begin{enumerate}
	\item[(i)] If $\beta>0$ and the rewards are constant (i.e. $w_{i,j} = w_{i',j'}$, for any $(i,j)$ and $(i',j')\in\maE$), then the matching policy is 
	`Match the Longest' ({\sc ml}), i.e., the incoming $v$-item is matched upon the arrival with an item of the compatible class having the longest queue size (ties being broken uniformly at random).
	\item[(ii)] If $\beta<0$ and the rewards are constant, then the matching policy is 
	`Match the Shortest' ({\sc ms}), i.e., the incoming $v$-item is matched upon the arrival with an item of the compatible class having the shortest queue size (ties being broken uniformly at random).
	\item[(iii)] If $\beta=0$ and $w_{i,j}\ne w_{i,j'}$ for any $i\in \maV$ and any $j\ne j' \in \maE(i)$ (implying that there is a strict ordering of rewards for all 
	possible matches of any given class), then the matching policy is of a priority type, defined above. 
	\end{enumerate}



\subsubsection{$\maV_2$-favorable policies}
An important class of matching policies is given by the following:

\begin{definition}
We say that an admissible matching policy $\Phi$ on $G$ is $\maV_2$-favorable if any incoming item always prioritizes a match with a compatible 
item of class in $\maV_2$ over a compatible item of class in $\maV_1$, whenever it has the choice. Formally, if the class-detail is given by $x\in\mathbb X$ and the arrival is of class $v$, it never occurs that the incoming $v$-item is matched with a $j$-item, for some $j\in\maV_1$, while $\maP(x,v) \cap \maV_2 \ne \emptyset$. 
\end{definition}

\subsection{Primary Markov representation}
\label{subsec:Markov}
The Markov representation of the model is similar {to} that of general matching models on graphs. 
Denote, for all $w\in\mathbb W$ and all $n\ge 1$, by $W^{\{w\}}_n$, the buffer-content at time $n$ 
(i.e. just after the arrival of the item $V_n$) if the buffer-content at time 0 was set to $w$. In other words,
\[\left\{\begin{array}{ll}
W^{\{w\}}_0 &=0,\\
W^{\{w\}}_n &= W^\Phi\left(wV_1\dots V_{n}\right),\quad n\in \N_+.\end{array}\right.\]
It readily follows from (\ref{eq:defodot}) and Remark \ref{rem:equiv} that the buffer-content sequence $\suite{W^{\{w\}}_n}$ is a Markov chain.
Indeed, for any $w\in\mathbb W$, we have 
\[W^{\{w\}}_{n+1} =W^{\{w\}}_n \odot_\Phi V_{n+1},\,\forall n\in\N.\]

For a fixed initial condition and for all $n\in\N$, we denote, for all $B \subset \maV$, by $W_n(B)$, the number of items in 
line of classes in $B$ at time $n$, and by $|W_n|$, the total number of items in the system at time $n$. In other words, 
\[\left\{\begin{array}{ll}
 W_n(B) &\isdef \displaystyle\sum_{i\in B} W_n(i),\\
 |W_n|  &\isdef W_n(\maV) = \displaystyle\sum_{i\in\maV} W_n(i). 
\end{array}\right.\]


In line with \cite{MaiMoy16}, for any admissible matching policy $\Phi$, we define the {\em stability region} associated to $G$ and $\Phi$ as the set of measures 
\begin{equation}
\label{eq:defstab}
\textsc{stab}(G,\Phi) \isdef \left\{\mu \in \mathscr M\left(\maV\right):\suite{W_n} \mbox{ is positive recurrent}\right\}.
\end{equation}
\begin{definition}
We say that a multigraph $G$ is stabilizable 
if $\textsc{stab}(G,\Phi)\neq\emptyset$, for some matching policy $\Phi$. {If not, $G$ is said non-stabilizable.}
\end{definition}

\begin{remark}\label{rmk:stab}\rm
	If $G$ is such that 
	$\maV = \maV_1$, i.e. all nodes of the multigraph have a self-loop, we say, for obvious reasons, that the considered matching models 
	are {\em finite}. Then any matching model on $G$ is necessarily stable, that is, for any admissible $\Phi$ we have that 
	\[\textsc{stab}(G,\Phi)=\mathscr M(\maV).\]
	Indeed, the Markov chain $\suite{W_n}$ is 
	irreducible on the finite state space $\mathbb W$, containing only words having size less or equal to the cardinality of the largest independent set of $G$. 
\end{remark}

\section{Main results}
\label{sec:results}
We now state the main results of this paper. 
Similarly to \cite{MaiMoy16}, we will be led to consider the set   
\begin{equation}
\label{eq:Ncond}
\textsc{Ncond}(G)\isdef \left\{\mu \in \maM\left(\maV\right):\forall\maI \in \I(G),\; \mu\left(\maI\right) < \mu\left(\maE\left(\maI\right)\right)\right\}. 
\end{equation}
Let us immediately observe that
\begin{lemma}
\label{lemma:Ncond}
For any connected multigraph $G$, we have that 
\[\textsc{Ncond}\left(\check{G}\right) \subseteq\textsc{Ncond}(G).\]
\end{lemma}
\begin{proof}
Let $\mu\in \textsc{Ncond}\left(\check{G}\right)$ and $\maI \in \I(G)$. Plainly, $\maI$ is then also an element of $\maI(\check G)$, so we have that 
$\mu(\maI) < \mu\left(\check{\maE}(\maI)\right)$. But as $\maI\subset\maV_2$ we also have that $\check{\maE}(\maI)=\maE(\maI)$, which concludes the proof. 
\end{proof}
It is stated in Theorem 1 of \cite{MaiMoy16} that, if $G$ is a graph, the set {\sc Ncond}$(G)$ is non-empty if and only if $G$ is not a bipartite graph. 
This result can be generalized to multigraphs, and completed as follows, 
\begin{proposition}
	\label{prop:ncond} 
	For any connected multigraph $G$, we have that 
	\[\textsc{Ncond}(G) \ne \emptyset \Longleftrightarrow \mbox{ $G$ is not a bipartite graph.}\]
	\noindent In that case, recalling (\ref{eq:nummudeg}), the probability measure $\mu_{\tiny{\mbox{deg}}}$ defined by
	\begin{equation}\label{eq:defmudeg}
	\mu_{\tiny{\mbox{deg}}}(i) = {\mbox{deg}(i) \over |\maE|},\quad i\in \maV,\end{equation}
	is always an element of {\sc Ncond}$(G)$.
\end{proposition}
\noindent Proposition \ref{prop:ncond} is proven in Section \ref{sec:ncond}. 
From Proposition 2 in \cite{MaiMoy16}, whenever $G$ is a graph (i.e., $\maV_1=\emptyset$), 
the set $\textsc{stab}(G,\Phi)$ is included in $\textsc{Ncond}(G)$ for any admissible policy $\Phi$. In other words, for any measure $\mu$, 
belonging to $\textsc{Ncond}(G)$ is necessary for the stability of the system $(G,\Phi,\mu)$, for any $\Phi$. 
A similar result holds for any multigraph $G$:

\begin{proposition}
	\label{thm:mainmono} 
	For any connected multigraph $G=(\maV,\maE)$ and any admissible matching policy
	$\Phi$, we have that 
	\[\textsc{stab}(G,\Phi) \subset \textsc{Ncond}(G).\]
\end{proposition}

\begin{proof}
	The proof is analog to that of Proposition 2 in \cite{MaiMoy16}. 
\end{proof}

\noindent Hence, the notion of maximality of a matching policy:

\begin{definition}
For any connected multigraph $G$ that is not a bipartite graph, a matching policy $\Phi$ is said maximal if the sets 
$\textsc{stab}(G,\Phi)$ and  $\textsc{Ncond}(G)$ coincide.
\end{definition}

Whenever $G$ is a graph, Theorem 1 of \cite{MBM17} shows, first, that the policy `First Come, First Matched' ({\sc fcfm}) is maximal, and second, that  the stationary probability of the chain $\suite{W_n}$ can be expressed in a remarkable product form. 
We generalize this result to multigraphs: 
\begin{theorem}
\label{thm:FCFM}
The matching policy `First Come, First Matched' is maximal: for any connected 
multigraph $G$ that is not a bipartite graph, we have that {\sc Stab}$(G,\textsc{fcfm})=\textsc{Ncond}(G)$.  
Moreover, for any $\mu\in\textsc{Ncond}(G)$ the unique stationary probability $\Pi_W$ of the chain $\left(W_n\right)_{n\in\mathbb{N}}$ is defined by 
\begin{equation*}
\left\{\begin{array}{ll}
\Pi_W(\varepsilon) &=\alpha;\\
\Pi_W(w)&=\alpha\displaystyle\prod\limits_{l=1}^q \frac{\mu(w_l)}{\mu(\mathcal{E}(\{w_1,\dots,w_l \}))},
\mbox{ for all }w=w_1\dots w_q\in\mathbb{W}\setminus\{\varepsilon\}, 
\end{array}\right.\end{equation*}
where
\begin{equation}
\label{eq:defalpha}
\alpha^{-1}\!\!=\!1+\!\sum\limits_{\mathcal{I}\in\mathbb{I}\left(\check G\right)}\sum\limits_{\sigma\in\mathfrak{S}_{|\mathcal{I}|}}\prod\limits_{i=1}^{|\mathcal{I}|} \frac{\mu\left(e_{\sigma(i)}\right)}{\mu(\mathcal{E}(\{e_{\sigma(1)},\dots,e_{\sigma(i)}\}))-\mu(\{e_{\sigma(1)},\dots,e_{\sigma(i)}\}\cap\maV_2)},
\end{equation}
and where we denote  $\maI=\{e_1,\dots,e_{|\mathcal{I}|}\}$ for any $\mathcal{I}\in\mathbb{I}\left(\check G\right)$. 
\end{theorem}
\noindent Theorem \ref{thm:FCFM} is proven in section \ref{sec:FCFM}. 

\begin{remark}
\label{rem:fini}\rm
If the model is finite, i.e. all nodes of $G$ have self-loops or in other words, $\maV_2=\emptyset$, then it readily follows from 
Theorem \ref{thm:FCFM} that the unique stationary probability on the finite state space $\mathbb W$, is given by 
\begin{equation*}
\left\{\begin{array}{ll}
\Pi_W(\varepsilon) &=\alpha;\\
\Pi_W\left(e_{\sigma(1)} \cdots e_{\sigma(|\maI|)}\right)&=\alpha\displaystyle\prod\limits_{i=1}^{|\mathcal{I}|} \frac{\mu\left(e_{\sigma(i)}\right)}{\mu(\mathcal{E}(\{e_{\sigma(1)},\dots,e_{\sigma(i)}\}))},\mbox{ for all }\maI\in \mathbb I(\check G),\,\sigma \in \mathfrak{S}_{|\mathcal{I}|},
\end{array}\right.\end{equation*}
with the normalizing constant 
\begin{equation*}
\alpha=\left[1+\!\sum\limits_{\mathcal{I}\in\mathbb{I}\left(\check G\right)}\sum\limits_{\sigma\in\mathfrak{S}_{|\mathcal{I}|}}\prod\limits_{i=1}^{|\mathcal{I}|} \frac{\mu\left(e_{\sigma(i)}\right)}{\mu(\mathcal{E}(\{e_{\sigma(1)},\dots,e_{\sigma(i)}\}))}\right]^{-1}\cdot 
\end{equation*}
\end{remark}

On another hand, as is shown in Theorem 3.2 of \cite{JMRS20}, all Max-Weight matching policies such that $\beta>0$ are maximal 
whenever $G$ is a graph. (Notice that Theorem 3.2 in \cite{JMRS20} is in fact shown only for $\beta=1$, however a generalization to any $\beta>0$ is straightforward, by modifying the rewards accordingly.) In particular, the maximality of `Match the Longest' ({\sc ml}) for GM models was first proven in Theorem 2 of \cite{MaiMoy16}, as a consequence of the corresponding result for EBM models, see Theorem 7.1 of \cite{BGM13}.   
This result can also be generalized to multigraphs: 
\begin{theorem}
\label{thm:ML}
Any Max-Weight policy $\Phi$ such that $\beta>0$ is maximal: for any multigraph $G$ that is not a bipartite graph, we have that 
{\sc stab}$(G,\Phi)=\textsc{Ncond}(G)$. 

\end{theorem}

\noindent Aside from {\sc fcfm} and Max-Weight policies, we can determine, or lower-bound, the stability region of the model for particular classes 
of multigraphs. 
\begin{definition} 
Let $G$ be a connected multigraph. 
We say that $G$ is complete $p$-partite, $p\ge 2$, if its maximal subgraph $\check G$ is complete $p$-partite. 
Then, the minimal blow-up graph $\hat G$ is itself called an {\em extended} complete $p$-partite graph. 
\end{definition}

Observe that an extended complete $p$-partite graph is {\em not} complete $p$-partite whenever the construction above is non-trivial, i.e. the multigraph in the above definition is not a graph, see an example on Figure \ref{fig:p-partiteCompletToBlow-upgraph}.
\begin{figure}[htb]
	\begin{center}
		\begin{tikzpicture}
		\fill (0,2) circle (2pt)node[above]{\scriptsize{1}};
		\fill (-1,0) circle (2pt)node[left]{\scriptsize{4}};
		\fill (1,0) circle (2pt)node[right]{\scriptsize{5}};
			\fill (-0.5,0.5) circle (2pt)node[above]{\scriptsize{$\;\;\;\;\;2$}};
		\fill (1.5,0.5) circle (2pt)node[right]{\scriptsize{3}};
		\draw[-] (1.5,0.5) -- (0,2);
		\draw[-] (-0.5,0.5) -- (0,2);
		\draw[-] (1,0) -- (0,2);
		\draw[-] (-1,0) -- (0,2);
		\draw[-] (-1,0) -- (1,0);
		\draw[-] (-1,0) -- (1.5,0.5);
		\draw[-] (-0.5,0.5) -- (1,0);
		\draw[-] (-0.5,0.5) -- (1.5,0.5);
		\draw[<-] (1.8,1) -- (2.2,1);
		\fill (4,2) circle (2pt)node[above]{\scriptsize{1}};
		\fill (3,0) circle (2pt)node[left]{\scriptsize{4}};
		\fill (5,0) circle (2pt)node[right]{\scriptsize{$\;5$}};
		\fill (3.5,0.5) circle (2pt)node[above]{\scriptsize{$\;\;\;\;\;2$}};
		\fill (5.5,0.5) circle (2pt)node[right]{\scriptsize{3}};
		\draw[-] (5.5,0.5) -- (4,2);
		\draw[-] (3.5,0.5) -- (4,2);
		\draw[-] (5,0) -- (4,2);
		\draw[-] (3,0) -- (4,2);
		\draw[-] (3,0) -- (5,0);
		\draw[-] (3,0) -- (5.5,0.5);
		\draw[-] (3.5,0.5) -- (5,0);
		\draw[-] (3.5,0.5) -- (5.5,0.5);
		\draw[thick,-] (5,0) to [out=-50,in=40,distance=11mm] (5,0);
		\draw[->] (5.8,1) -- (6.2,1);
		\fill (8,2) circle (2pt)node[above]{\scriptsize{1}};
		\fill (7,0) circle (2pt)node[left]{\scriptsize{4}};
		\fill (7.5,0.5) circle (2pt)node[above]{\scriptsize{$\;\;\;\;\;2$}};
		\fill (9,0) circle (2pt)node[right]{\scriptsize{5}};
		\fill (9.5,0.5) circle (2pt)node[right]{\scriptsize{3}};
		\fill (8,-0.5) circle (2pt)node[right]{\scriptsize{$\;\,\cop{5}$}};
		\draw[-] (9.5,0.5) -- (8,2);
		\draw[-] (7.5,0.5) -- (8,2);
		\draw[-] (9,0) -- (8,2);
		\draw[-] (7,0) -- (8,2);
		\draw[-] (7,0) -- (9,0);
		\draw[-] (7,0) -- (9.5,0.5);
		\draw[-] (7.5,0.5) -- (9,0);
		\draw[-] (7.5,0.5) -- (9.5,0.5);
		\draw[-] (8,-0.5) -- (9,0);
		\draw[-] (8,-0.5) -- (7.5,0.5);
		\draw[-] (8,-0.5) -- (7,0);
		\draw[-] (8,-0.5) -- (8,2);
		\end{tikzpicture}
		\caption{A multigraph $G$ (middle), its maximal complete $3$-partite subgraph $\check G$ (left), and extended complete $p$-partite graph $\hat{G}$ (right).}
		\label{fig:p-partiteCompletToBlow-upgraph}
	\end{center}
\end{figure}
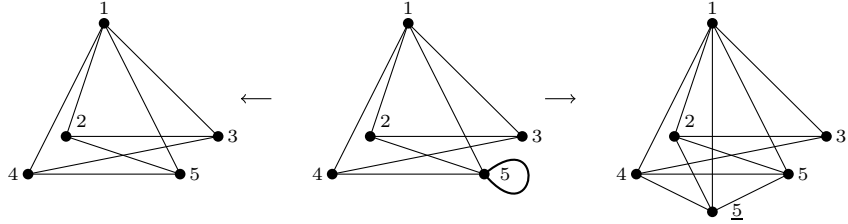

\begin{theorem}
	\label{thm:ppartite}
	Let $G$ be a complete $p$-partite multigraph, $p\ge 2$. Then, 
	\begin{enumerate}
		\item[(i)] If $p\ge 3$ or $\maV_1\neq \emptyset$, then any $\maV_2$-favorable matching policy $\Phi$ is maximal, that is, $\textsc{stab}(G,\Phi)=\textsc{Ncond}(G).$
        \item[(ii)] If $p\ge 3$, then $\textsc{Ncond}\left(\check G\right)\subset \textsc{stab}(G,\Phi)$, for any admissible matching policy $\Phi$. 
  \end{enumerate}  
   \end{theorem}
   
With the above results in hands, we have the following panorama regarding the stability region of a matching model on a connected multigraph $G$: 
\begin{enumerate}
\item[(i)] Any measure $\mu$ that does not belong to the set {\sc Ncond}($G$) makes the system {unstable}; 
\item[(ii)] If $G$ is a bipartite graph, then the model cannot be stable;
\item[(iii)] Otherwise, the region {\sc Ncond}($G$) is necessarily non-empty,  
and the models $(G,\textsc{fcfm},\mu)$ and $(G,\Phi,\mu)$ for any Max-Weight policy $\Phi$ with $\beta>0$, are stable for any $\mu\in\textsc{Ncond}(G)$. In particular, 
in that case  {\sc Ncond}($G$) necessarily includes the measure $\mu_{\tiny{\mbox{deg}}}$ defined by (\ref{eq:defmudeg}), so 
any model $\left(G,\Phi,\mu_{\tiny{\mbox{deg}}}\right)$, for $\Phi=\textsc{fcfm}$ or $\Phi=\textsc{mw}$ with $\beta>0$, is stable.  
\item[(iv)] For any complete $p$-partite multigraph ($p\ge 2$) that is not a bipartite graph and any  $\mu\in\textsc{Ncond}\left({G}\right)$, 
any model $(G,\Phi,\mu)$ such that $\mu\in\textsc{Ncond}\left(\check{G}\right)$ or $\Phi$ is $\maV_2$-favorable, is stable. 
\end{enumerate} 

As a by-product of Theorem \ref{thm:ppartite} we can determine, or lower-bound, the stability region of GM models on extended complete 
$p$-partite graphs.

\begin{definition}
\label{def:extendsmeas}
For any measures $\mu\in\mathscr M(\maV)$ and $\hat\mu\in\mathscr M(\hat\maV)$, we say that $\hat \mu$ extends $\mu$ on $\hat G$, and that 
$\mu$ reduces $\hat \mu$ on $G$, if 
 \[\left\{\begin{array}{ll}
	\displaystyle\hat{\mu}(i)&=\mu(i),\quad \text{for all} \; i\in\maV_2\,;\\
	\displaystyle\hat{\mu}(i) + \hat{\mu}(\underline{i}) &=\mu(i),\quad \text{for all} \; i\in\maV_1.
	\end{array}\right.
\]
\end{definition}


\begin{definition}
\label{def:extendspol}
Let $\Phi$ and $\hat\Phi$ be two admissible matching policies, respectively on $G$ and $\hat G$. 
We say that $\hat\Phi$ extends $\Phi$ on $\hat G$ if, for any $\mu\in\mathscr M({\maV})$ and $\hat\mu\in\mathscr M\left(\hat\maV\right)$, whenever both systems $(G,\Phi,\mu)$ and $\left(\hat G,\hat\Phi,\hat\mu\right)$ are in the same state $w\in \mathbb W$ and welcome the same arrival, $\Phi$ and $\hat\Phi$ induce the same choice of match, if any. 
\end{definition}

\begin{proposition}
	\label{prop:extppartite}
	Let $\hat G$ be an extended complete $p$-partite graph, $p\ge 2$, and $\check G$ be its reduced graph. 
	\begin{enumerate}
	\item[(i)] If $p\ge 3$ or $\maV_1\neq \emptyset$, then for any matching policy $\hat \Phi$ on $\hat G$ that extends a $\maV_2$-favorable policy on $G$, 
		$\textsc{stab}(\hat G,\hat \Phi)=\textsc{Ncond}(\hat G).$
        \item[(ii)] If $p\ge 3$, then for any measure $\hat\mu$ on $\hat G$ whose reduced measure $\mu$ is an element of $\textsc{Ncond}(\check G)$,
         and any matching policy $\hat{\Phi}$ on $\hat G$, the model $(\hat G,\hat \Phi,\hat \mu)$ is stable. 
  \end{enumerate}  
   \end{proposition}
   
\noindent The proofs of Theorem \ref{thm:ML}, Theorem \ref{thm:ppartite} and Proposition \ref{prop:extppartite} are given in section \ref{sec:otherproofs}. 

\section{A few examples}
\label{sec:examples}

In this section, we illustrate our main results by different examples.

\begin{ex}\label{ExampleSquare}\rm Consider the multigraph $G$ of Figure \ref{fig:ExampleSquare}, made of four nodes arranged in a square, wit a self-loop at each node.
\begin{figure}[htb]
\begin{center}	
\begin{tikzpicture}
\draw (-1,1)--(1,1) ;
\draw (1,1)--(1,-1) ;
\draw (1,-1)--(-1,-1) ;
\draw (-1,-1)--(-1,1) ;

\draw[thick,-] (-1,1) to [out=130,in=240,distance=15mm] (-1,1);
\draw [thick,-] (1,1) to [out=50,in=-50,distance=15mm] (1,1);
\draw[thick,-] (-1,-1) to [out=130,in=240,distance=15mm] (-1,-1);
\draw [thick,-] (1,-1) to [out=50,in=-50,distance=15mm] (1,-1);

\fill (-1,1) circle (2pt)node[left]{$1\;$} ;
\fill (1,1) circle (2pt)node[right]{$\;2$} ;
\fill (1,-1) circle (2pt)node[right]{$\;3$} ;
\fill (-1,-1) circle (2pt)node[left]{$4\;$} ;
\end{tikzpicture}
\caption{Multigraph $G$ of Example~\ref{ExampleSquare}.}\label{fig:ExampleSquare}
\end{center}
\end{figure}
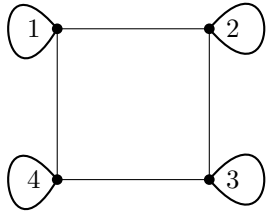
Since all nodes have a self-loop, it follows from Remark~\ref{rmk:stab} that any matching model on $G$ is necessarily stable, that is, for any admissible $\Phi$, we have that $\textsc{stab}(G,\Phi)=\mathscr M(\maV).$ Let us focus on the $\textsc{fcfm}$ policy. 
The set of admissible queue details is given by $\mathbb W = \{\varepsilon,1,2,3,4,13,24,31,42\},$ and as a consequence of Remark~\ref{rem:fini}, we can compute explicitly $\Pi_W$, obtaining the following values:
\[\left\{\begin{array}{llll}
& \Pi_W(\varepsilon)=\alpha && \\
& \Pi_W(1)=\alpha\frac{\mu(1)}{1-\mu(3)} &&
\Pi_W(2)=\alpha\frac{\mu(2)}{1-\mu(4)} \\
& \Pi_W(3)=\alpha\frac{\mu(3)}{1-\mu(1)} &&
\Pi_W(4)=\alpha\frac{\mu(4)}{1-\mu(2)} \\
& \Pi_W(13)=\alpha\frac{\mu(1)}{1-\mu(3)}\mu(3) &&
\Pi_W(24)=\alpha\frac{\mu(2)}{1-\mu(4)}\mu(4) \\
& \Pi_W(31)=\alpha\frac{\mu(3)}{1-\mu(1)}\mu(1) &&
\Pi_W(42)=\alpha\frac{\mu(4)}{1-\mu(2)}\mu(2),
\end{array}\right.\]
with
$$
\alpha=\left[1 + \mu(1)\frac{1+\mu(3)}{1-\mu(3)}
+ \mu(2)\frac{1+\mu(4)}{1-\mu(4)}
+ \mu(3)\frac{1+\mu(1)}{1-\mu(1)}
+ \mu(4)\frac{1+\mu(2)}{1-\mu(2)} \right]^{-1},
$$
using the fact that $\mathbb{I}\left(\check G\right)=\{\{1\},\{2\},\{3\},\{4\},\{1,3\},\{2,4\}\}$.
\end{ex}

\begin{ex}\rm
Consider the multigraph $G$ (at the middle) of Figure~\ref{fig:GgraphGLmultigZAndGTilde}. 
From Theorems \ref{thm:FCFM} and \ref{thm:ML}, both the stability region $\textsc{stab}(G,\textsc{fcfm})$ under First Come, First Matched, and the stability region $\textsc{stab}(G,\textsc{mw})$ 
under any Max-Weight policy, coincide with the set 
\begin{align*}
\textsc{Ncond}(G) &=\left\{\mu\in\mathscr M(\maV):\mu(1) < \mu(2),\,\mu(\{1,3\})\vee\mu(\{1,4\})<{1\over 2}\right\}.
\end{align*}
\end{ex}

\begin{ex}\rm
	Consider now the multigraph $G$ of Figure~\ref{fig:ExGLandGTilde}, whose maximal subgraph is a complete $2$-partite graph of {order} 3 (i.e., a string of 3 nodes). 
	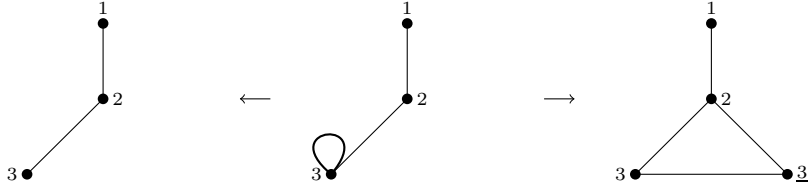
\begin{figure}[htb]
	\begin{center}
		\begin{tikzpicture}
			\fill (0,2) circle (2pt)node[above]{\scriptsize{1}};
		\fill (0,1) circle (2pt)node[right]{\scriptsize{2}};
		\fill (-1,0) circle (2pt)node[left]{\scriptsize{3}};
		\draw[-] (0,1) -- (-1,0);
		\draw[-] (0,1) -- (0,2);
		\draw[<-] (1.8,1) -- (2.2,1);
		\fill (4,2) circle (2pt)node[above]{\scriptsize{1}};
		\fill (4,1) circle (2pt)node[right]{\scriptsize{2}};
		\fill (3,0) circle (2pt)node[left]{\scriptsize{3}};
		\draw[-] (4,1) -- (3,0);
		\draw[-] (4,1) -- (4,2);
		\draw[->] (5.8,1) -- (6.2,1);
		\fill (8,2) circle (2pt)node[above]{\scriptsize{1}};
		\fill (8,1) circle (2pt)node[right]{\scriptsize{2}};
		\fill (7,0) circle (2pt)node[left]{\scriptsize{3}};
		\fill (9,0) circle (2pt)node[right]{\scriptsize{$\cop{3}$}};
		\draw[-] (8,1) -- (9,0);
		\draw[-] (8,1) -- (7,0);
		\draw[-] (8,1) -- (8,2);
		\draw[-] (7,0) -- (9,0);
		\draw[thick,-] (3,0) to [out=50,in=140,distance=10mm] (3,0);
		\end{tikzpicture}
		\caption{A multigraph on the complete 2-partite graph of size 3 (left), and its minimal blow-up graph (right).}
		\label{fig:ExGLandGTilde}
	\end{center}
\end{figure}

\noindent We easily obtain that 
\begin{align*}
\textsc{Ncond}(G)&=\left\{\mu\in \mathscr M(\maV)\,:\, \mu(1) <\mu(2)<{1 \over 2}\right\},\\
\textsc{Ncond}(\hat G) &= \left\{\mu\in\mathscr M(\hat\maV):\mu(1) < \mu(2),\,\mu(2)\vee\mu(\{1,3\})\vee\mu(\{1,\cop{3}\})<{1\over 2}\right\}.
\end{align*}
In view of Theorems \ref{thm:FCFM} and \ref{thm:ML}, the respective stability regions $\textsc{stab}(G,\textsc{fcfm})$ and $\textsc{stab}(G,\textsc{mw})$ under First Come, First Matched, 
or any Max-Weight policy, coincide with $\textsc{Ncond}(G)$.  

Let us first focus on the \textsc{fcfm} policy. The set of admissible queue details is given by
$$\mathbb{W}=\{\varepsilon\}\cup\left\{1^k : k\geq 1\right\}\cup\left\{2^k : k\geq 1 \right\}\cup\left\{1^r 3 1^{k-r} : k\geq 0, \; 0\leq r\leq k \right\}.$$
By Theorem~\ref{thm:FCFM}, we have that 
\[\left\{\begin{array}{ll}
&\Pi_W(\varepsilon)=\alpha \\
&\Pi_W(1^k)=\alpha\left(\frac{\mu(1)}{\mu(2)}\right)^k \qquad \Pi_W(2^k)=\alpha\left(\frac{\mu(2)}{1-\mu(2)}\right)^k \\
&\Pi_W\left(1^r 3 1^{k-r}\right)= \alpha\left(\frac{\mu(1)}{\mu(2)}\right)^r \times \frac{\mu(3)}{1-\mu(1)} \times \left(\frac{\mu(1)}{1-\mu(1)}\right)^{k-r},
\end{array}\right.\]
and since $\mathbb{I}\left(\check G\right)=\{\{1\},\{2\},\{3\},\{1,3\}\}$, we can express $\alpha$ as follows,
\begin{multline*}
\alpha=\biggl[1 + \frac{\mu(1)}{\mu(2)-\mu(1)} + \frac{\mu(2)}{1-2\mu(2)} + \frac{\mu(3)}{1-\mu(1)} \\
+ \frac{\mu(1)}{\mu(2)-\mu(1)}\frac{\mu(3)}{1-2\mu(1)}+\frac{\mu(3)}{1-\mu(1)}\frac{\mu(1)}{1-2\mu(1)} \biggl]^{-1}.\end{multline*}

Second, consider a matching policy $\hat\Phi$ such that a $2$-item always prioritizes a $1$-item over a $3$ or a $\cop{3}$-item. Then $\hat\Phi$ extends a $\maV_2$-favorable policy 
$\Phi$ on $\hat G$. Thus, from Proposition \ref{prop:extppartite}(i), the stability region of the system is $\textsc{Ncond}(\hat G)$, 
in other words $\hat\Phi$ is maximal on $\hat G$. We thereby generalize with a very simple proof, the result of Lemma 3 of \cite{MaiMoy16} to the case where $\mu(3) \ne \mu(\cop{3})$. 
Last, in view of Theorem \ref{thm:ppartite}(i), any $\maV_2$-favorable matching policy $\Phi$ on $G$ (i.e., such that $2$ prioritizes $1$ over $3$) is maximal, that is, the two sets $\textsc{stab}(G)$ and $\textsc{Ncond}(G)$ coincide. 
\end{ex}

\begin{ex}\rm
Last, consider the multigraph $G$ represented in (the middle figure of) Figure 
\ref{fig:p-partiteCompletToBlow-upgraph}. The maximal subgraph is $3$-partite complete, and we readily obtain that 
\begin{align*}
\textsc{Ncond}(\check G) &=\left\{\mu\in \mathscr M(\maV)\,:\, \mu(1) \vee \mu(\{2,4\}) \vee \mu(\{3,5\}) <{1\over 2}\right\};\\
\textsc{Ncond}(G) &=\left\{\mu\in \mathscr M(\maV)\,:\, \mu(1) \vee \mu(\{2,4\})<{1\over 2},\;\mu(3) < \mu(\{1,2,4\}) \right\};\\
\textsc{Ncond}(\hat G) &=\left\{\mu\in \mathscr M(\maV)\,:\, \mu(1) \vee \mu(\{2,4\}) \vee \mu(\{3,5\}) \vee \mu(\{3,\cop{5}\}) <{1\over 2}\right\}.
\end{align*}
Then, from Theorems \ref{thm:FCFM} and \ref{thm:ML}, the respective stability regions $\textsc{stab}(G,\textsc{fcfm})$ and $\textsc{stab}(G,\textsc{mw})$ under First Come, First Matched, 
or any Max-Weight policy coincide with the set $\textsc{Ncond}(G)$. From Theorem \ref{thm:ppartite}(i), for any policy $\Phi$ on $G$ according to which all items prioritize $3$-items over $5$ items is maximal, i.e. 
$\textsc{stab}(G,\Phi)=\textsc{Ncond}(G)$. From Theorem \ref{thm:ppartite}(ii), any policy $\Phi$ on $G$ is such that $\textsc{Ncond}(\check G)\subset \textsc{stab}(G,\Phi)$. 
Last, from Proposition \ref{prop:extppartite}, any policy $\hat\Phi$ on $\hat G$ giving priority to $3$-items over $5$ and $\cop{5}$-items is maximal, whereas for any matching policy 
$\hat\Phi$ and any measure $\hat\mu$ on $\hat\maV$ extending a measure of $\textsc{Ncond}(\check G)$, the model $(\hat G,\hat\Phi,\hat\mu)$ is stable. 
\end{ex}

\section{Proof of Proposition \ref{prop:ncond}}\label{sec:ncond}

%
The proof of the fact that $\textsc{Ncond}(G)=\emptyset$ whenever $G$ is a bipartite graph was given in Theorem 1 of \cite{MaiMoy16}. We reproduce it hereafter for easy reference: 
Let us assume that $G$ is a connected bipartite graph, and let $(A,B)$ be a bipartition of $\maV$ such that if $i\v j$, then $(i,j)\in (A\times B)\cup (B\times A)$. Then, $A, B\in \I(G)$, and $\maE(A)=B, \maE(B)=A$. So, for any measure $\mu \in \mathscr M(\maV)$, we have either $\mu(A) \geq \mu(\maE(A))$ or  $\mu(B) \geq \mu(\maE(B))$, meaning that $\textsc{Ncond}(G)=\emptyset.$

\medskip

In order to prove the converse statement,  
we show that if $G$ is a connected multigraph such that $\mu_{\tiny{\mbox{deg}}}\not\in\textsc{Ncond}(G)$, then $G$ is a bipartite graph. 

For a subset $A$ of $\maV$, let us denote by $\out(A)$ the set of edges having at least one extremity in $A$, where we identify the edges $(i,j)$ and $(j,i)$, and where we include self-loops. Formally, 
$$\out(A)\isdef \{\{i,j\}\subset \maV : i\ne j,\,i \v j \mbox{ and } \{i,j\} \cap A\not=\emptyset\}\cup \{(i,i),\,i\in\maV_1\cap A\}.$$ 
Observe that we have the following properties.
\begin{enumerate}
\item For any $A\subset \maV$, $\out(A)\subset\out(\maE(A))$. 
\item For any $A\subset \maV$,
$\sum_{i\in A} \deg(i)\geq \card \out(A)$, so that $\mudeg(A)\geq {1\over |\maE|} \card \out(A)$, with equality if $A \in \I(G)$.
\end{enumerate}

Let us assume that $\mudeg\not\in\textsc{Ncond}(G)$, meaning that there exists $\maI \in \I(G)$ with $\mudeg(\maI) \geq \mudeg(\maE(\maI))$. By the property (2) above, we have 
$\mudeg(\maI)={1\over |\maE|} \card\out(\maI)$ and $\mudeg(\maE(\maI))\geq{1\over |\maE|} \card\out(\maE(\maI))$, 
so that $\card\out(\maI)\geq \card\out(\maE(\maI))$. But by property (1), $\out(\maI)\subset \out(\maE(\maI))$, which implies that $\out(\maI)=\out(\maE(\maI))$. 

Let us prove that $(\maI,\maE(\maI))$ is a bipartition of $\maV$ such that if $i\v j$, then $(i,j)\in (\maI\times \maE(\maI))\cup (\maE(\maI)\times\maI)$. 
\begin{itemize}
\item Since $\maI \in \I(G),$ $\maI\cap \maE(\maI)=\emptyset$. 
\item Using that $\out(\maI)=\out(\maE(\maI))$, we deduce that $\maE(\maE(\maI))=\maI$. So, the vertices reachable from $\maI$ all belong to $\maI\cup \maE(\maI)$. The multigraph $G$ being connected, this implies that $\maI\cup\maE(\maI)=\maV$.
\item Since $\maI \in \I(G),$ there is no edge between vertices of $\maI$. Furthermore, since $\out(\maI)=\out(\maE(\maI))$, there is no edge between vertices of $\maE(\maI)$.
\end{itemize}
Consequently, $G$ is a bipartite graph, which ends the proof.


\section{Proof of Theorem \ref{thm:FCFM}}
\label{sec:FCFM}

Let us recall that the multigraph $G=(\mathcal{V},\mathcal{E})$ is connected but is not a bipartite graph, with $|\mathcal{V}|\geq 2$. Then, in particular, $\textsc{Ncond}(G)\neq \emptyset$ (cf. Proposition \ref{prop:ncond}). Our product form result, Theorem~\ref{thm:FCFM}, follows from a reversibility scheme that generalizes to the case of multigraphs, the one constructed in \cite{MBM17}. In fact, we propose a proof that is simpler, at some points, than the one in \cite{MBM17}. We reproduce hereafter the main steps of this construction for easy reference, and only develop exhaustively the points that are specific to the present context, or based on different arguments. 

Hereafter, we denote by $\P_W$, the transition operator of the buffer-content Markov chain, that is, for all $w,w' \in\mathbb W,$ 
we write $P_W(w,w')=\pr{W_{n+1}=w' \mid W_n =w}, $ for any $n\in \N$. 


\subsection{Two auxiliary chains}
\label{subsec:aux}
As in section 3.2 of \cite{MBM17}, we first need to define two auxiliary Markov chains. For this, let us denote by $\oV$ an independent copy of $\V$, i.e. a set with the same cardinal formed with copies of elements of $\V$. We set $\bV\isdef \V\cup\oV,$ and we define, for $\mathbf{w}\in\bV^*$,
\begin{align*}
&\V(\w )\isdef \{a\in\V : |\w |_a>0\},\\
&\oV(\w )\isdef \{\overline{a}\in\oV : |\w |_{\overline{a}}>0\}.
\end{align*}
For $a\in\bV,$ we will use the notation $\overline{\overline{a}}=a.$ 
\begin{definition} We define the \emph{backward detailed chain} as the process $(B_n)_{n\in\N}$ with values in $\bV^*$ given by 
$B_0=\varepsilon$ and, for any $n\geq 1$,
\begin{itemize}
\item if $W_n=\varepsilon$ (i.e. all the items arrived up to time $n$ are matched at time $n$), then $B_n=\varepsilon$,
\item otherwise, let $i(n)\in[\![1,n]\!]$ be the arrival time of the oldest item still in the buffer, then, the word $B_n$ is the word of length $n-i(n)+1$, defined, for any $\ell \in \ll 1,n-i(n)+1 \rr$, by
\[(B_n)_\ell=\left\{\begin{array}{ll}
V_{i(n)+\ell-1} \,\, &\mbox{if $V_{i(n)+\ell-1}$ has not been matched up to time $n$};\\
\td{V_{k}}\,\, &\mbox{if $V_{i(n)+\ell-1}$ is matched at or before time $n$, with item $V_k$}\\
               &\mbox{(where $1 \leq k \le n$)}.\\
\end{array}\right.\] 

\end{itemize}
\end{definition}

In other words, 
the word $B_n$ gathers the class indexes of all unmatched items entered up to $n$, at the places corresponding to their arrival times, 
and the copies of the class  
indexes of the items matched before $n$, but after the arrival of the oldest unmatched item at $n$, at the place corresponding to the arrival time of 
their respective match. 


Observe that by construction of $(B_n)_{n\in\N}$, for all $n\in\N$, the word $B_n$ necessarily contains all the letters of $W_n$. More precisely, for any $n\in\mathbb{N}$, $W_n$ is the restriction of the word $B_n$ to its letters in $\V$. 
Furthermore, $(B_n)_{n\in\N}$ is also a Markov chain, since for any $n\geq 0$, the value of $B_{n+1}$ can be deduced from that of $B_n$ and from the class $V_{n+1}$ of the item entered at time $n+1$.  

A state $\w \in \bV^*$ is said to be \emph{admissible for $(B_n)_{n\in\N}$} if it can be reached by the chain $(B_n)_{n\in\N}$, under the $\fcfm$ policy. We set
$$\mathbb{B}\isdef \{\w \in\bV^* : \w  \; \text{is admissible for} \; (B_n)_{n\in\N}\}.$$
The following result can be proven exactly as Lemma 1 in \cite{MBM17}.
\begin{lemma}\label{lem:ADM_B}
Let $\w =\w _1\dots \w _q \in \bV^*.$ Then, $\w\in\mathbb{B}$ if and only if $\w_1\in\V$ and for $1\leq i<j\leq q$,
\begin{itemize}
\item if $(\w_i,\w_j)\in \V^2$, then $\w_i\pv\w_j$,
\item if $(\w_i,\w_j)\in \V\times\oV$, then $\w_i\pv{\overline \w_j}$.
\end{itemize}
\end{lemma}

As a consequence of Lemma~\ref{lem:ADM_B}, any word $\w \in \mathbb{B}$ can be written as
$$
\w=b_1\overline{a_{11}}\overline{a_{12}}\dots\overline{a_{1k_1}}b_2\overline{a_{21}}\overline{a_{22}}\dots \overline{a_{2k_2}} b_3 \dots b_q \overline{a_{q1}} \dots \overline{a_{qk_q}},$$
where $q, k_1,\ldots, k_q\in\N,$ $b_1,\ldots,b_q\in\V$, $a_{ij}\in\V$ for $1\leq i\leq q, 1\leq j\leq k_i$, and
\[\left\{\begin{array}{ll}
\{ b_1, \dots, b_q \}=\V(\w )\in\mathbb{I}\left(\check G\right),\\
\forall i\in\ll 1,q\rr  , \; b_i\in\S \; \Rightarrow \left[\forall j\neq i, \; b_i\neq b_j\right], \\
\forall i \in \ll 1,q\rr  , \; \forall j\in \ll 1,k_i\rr  , \; a_{ij}\in\E(\{b_1,\dots,b_i\})^c.
\end{array}\right.\]
The transition operator of the chain $\suite{B_n}$ is denoted by $\P_B$, that is, for all $\bw,\bw' \in \mathbb B$, we write 
$\P_B(\bw,\bw')=\pr{B_{n+1}=\bw' \mid B_n=\bw}$, for all $n\in \N$. 

\begin{definition}
We define the \emph{forward detailed chain} as the process $(F_n)_{n\in\N}$ with values in $\bV^*$ given by $F_0=\varepsilon$ (the empty word) and, for any $n\geq 1$,
\begin{itemize}
\item if $W_n=\varepsilon$ (i.e. all the items arrived up to time $n$ are matched at time $n$), then $F_n=\varepsilon$,
\item otherwise, let $\mathscr U_n$ be the set of items arrived before time $n$ that are not matched at time $n$ (note that $\mathscr U_n$ is non-empty, since $W_n \neq \varepsilon$). Also, set $$j(n)\isdef\sup\left\{ m \geq n+1 : V_m \mbox{ is matched with an element of } \mathscr U_n\right\}.$$
Observe that $j(n)$ is possibly infinite. Then, 
if $j(n)$ is finite, $F_n$ is the word of $\bV^*$ of length $j(n)-n$ (respectively of $A^{\N}$ of length $+\infty$, if $j(n)=+\infty$), 
such that for any $\ell \in \ll 1,j(n)-n \rr$ (respectively $\ell \in \N_+$), 
      \[(F_n)_\ell=\left\{\begin{array}{ll}
                         V_{n+\ell} \,\, &\mbox{if $V_{n+\ell}$ is not matched with an item arrived up to $n$};\\
                        \td{V_{k}}\,\, &\mbox{if $V_{n+\ell}$ is matched with item $V_k$, where $1 \leq k \le n$}.
                        \end{array}\right.\] 
\end{itemize}
\end{definition}


In other words, the word $F_n$ contains the copies of all the class indexes of the items entered up to time $n$ and matched after $n$, at the place corresponding to the arrival time of their respective match, 
together with the class indexes of all items entered after $n$ and before the last item matched with an item entered up to $n$, and not matched with an element entered before $n$, if any, at the place corresponding to their arrival time.  
Similarly to \cite{MBM17}, we make the three following simple observations: 
\begin{itemize}
\item If $F_n \in \bV^*$ is finite, then $(F_n)_{j(n)-n} \in \td\maV$;
\item $\suite{F_n}$ is a Markov chain; 
\item If $F_n$ is a.s. an element of $\mathbf V^*$ for all $n\in\N$.
\end{itemize} 


As for the backward chain, we say that a state $\w \in \bV^*$ is \emph{admissible for $(F_n)_{n\in\N}$} if it can be reached by the chain $(F_n)_{n\in\N}$, under the $\fcfm$ policy. Then, we set 
$$\mathbb{F}\isdef \{\w \in\bV^* : \w  \; \text{is admissible for} \suite{F_n}$$
and we denote by $\P_F$ the transition operator of the chain $\suite{F_n}$ on $\mathbb F$.  
For any word $\w =\w _1\dots\w _n\in\bV^*$, let us define its reversed-copy by $\overleftarrow{\overline{\w }}\isdef \overline{\w _n}\dots \overline{\w _1}\in\bV^*.$ Note that the map 
$\Psi: \bV^* \to \bV^*, \w \mapsto \overleftarrow{\overline{\w}}$ satisfies $\Psi\circ\Psi = Id_{\bV^*}$. Thus, $\Psi$ is a bijection and its inverse function is $\Psi^{-1}=\Psi$. 

\begin{lemma}\label{lem:ADMB_ADMF_iso}
The map 
\[\Phi:\begin{cases} {\mathbb B} \longrightarrow {\mathbb F}\\
 \w \longmapsto \Psi(\w)=\overleftarrow{\overline{\w}}\end{cases}\]
  is well-defined and bijective.
\end{lemma}

\begin{proof}
Exactly as in Lemma 2 in \cite{MBM17}, it can be proven that $\bw\in\mathbb F$ if and only if $\Psi(\bw)\in \mathbb B$. 
This guarantees that the mapping $\Phi$ is well-defined and surjective. It is injective because $\Psi$ clearly is so. 
\end{proof}

Let us define a measure $\nu$ on $\bV^*$ by $\nu(\varepsilon)\isdef 1$ and
\begin{equation}
\label{eq:defnu}
\forall \w \in \bV^*\setminus\{\varepsilon\}, \; \nu(\w ) \isdef \prod\limits_{i=1}^{|\V|} \mu(i)^{|\w |_i + |\overline{\w}|_i}.
\end{equation}

We can use the measure $\nu$ defined above to establish the following link between the dynamics of the chains $(B_n)_{n\in\N}$ and $(F_n)_{n\in\N}$. 
The following result can be established exactly as Lemma 3 in \cite{MBM17},


\begin{proposition}\label{prop:lienB_nF_n} For any $(\w ,\mathbf{w'})\in \mathbb{B}^2$, 
we have that 
\begin{equation*}\label{eq:lienB_nF_n}
\nu(\w )\P_B(\w,\w') = \nu \left(\overleftarrow{\overline{\mathbf{w'}}}\right) \P_F\left(\overleftarrow{\overline{\w' }},\overleftarrow{\overline{\mathbf{w}}}\right).
\end{equation*}
\end{proposition}


\subsection{Positive recurrence of $(B_n)_{n\in\N}$ and $(F_n)_{n\in\N}$.}
We will exploit the local balance equations of Proposition \ref{prop:lienB_nF_n} to derive stationary distributions of these two Markov chains. 
To this end, the following technical lemma will simplify the proofs.

\begin{lemma}\label{lem:nu(A*)}
The measure $\nu$ defined by (\ref{eq:defnu}) satisfies the following properties:
\begin{enumerate}
\item For any $\A \subset \bV = \V\cup\oV$, we have $\nu(\A )=\mu(\V(\A )) + \mu\left(\overline{\oV(\A )}\right)$. 
\item For any $\A _1,\ldots, \A _n\subset \bV, \; \nu(\A _1\dots\A _n)=\nu(\A _1)\dots\nu(\A _n)$. 
In particular, $\nu(\A ^k)=\nu(\A )^k$.
\item If $\A \subset \bV$ is such that $\nu(\A )<1$, then $\nu(\A ^*)={1\over 1-\nu(\A )}$.
\end{enumerate}
\end{lemma}

\begin{proof} The first point follows from the definition of $\nu$ and the second point is a direct consequence of its multiplicative structure. 
Regarding the third point, observe that $\A^*=\cup_{k\in\N} \A^k$, so that 
\[\nu(\A^*)=\sum_{k\in\N}\nu\left(\mathcal{A}^k\right)=\sum_{k\in\N}\nu(\mathcal{A})^k.\]
\end{proof}
\noindent We can now state the following result,
\begin{proposition}\label{prop:B_nF_n_rec_pos}
Suppose that $\mu\in\textsc{Ncond}(G)$. Then, 
 the chains $(B_n)_{n\in\N}$ and $(F_n)_{n\in\N}$ are positively recurrent and admit respectively the restrictions on $\mathbb B$ and on $\mathbb F$ 
 of $\nu$ (that is, $\nu_B(\w ) = \nu_F(\Phi(\w )) =\nu(\w)$, for any $\w\in\mathbb B$)  as unique stationary measure (up to a multiplicative constant), respectively on $\mathbb B$ and $\mathbb F$. 
\end{proposition}

\begin{proof}
Let $\mu\in\textsc{Ncond}(G)$.\\
\underline{Step 1}: we first prove that $\nu_B$ is a stationary measure for the chain $(B_n)_{n\in\N}.$\\
For this, let us fix $\mathbf{w'}\in \mathbb{B}$. Then we have that 
\begin{align*}
\displaystyle\sum\limits_{\w \in \mathbb{B}} \frac{\P_B(\w,\w')\nu_B(\w )}{\nu_B(\mathbf{w'})}
&= \sum\limits_{\w \in \mathbb{B}}
\frac{\P_F\left(\overleftarrow{\overline{\w'}},\overleftarrow{\overline{\mathbf{w}}}\right)\nu_B \left(\overleftarrow{\overline{\mathbf{w'}}}\right)}{\nu_B(\mathbf{w'})}\\
&= \sum\limits_{\w \in \mathbb{B}}
\P_F\left(\overleftarrow{\overline{\w'}},\overleftarrow{\overline{\mathbf{w}}}\right)= 1,
\end{align*}
where the first equality follows from Proposition ~\ref{prop:lienB_nF_n}, 
the second from the fact that 
$\nu_B \left(\overleftarrow{\overline{\mathbf{w'}}}\right)=\nu_B(\mathbf{w'})$ and the last, from Lemma~\ref{lem:ADMB_ADMF_iso}.  
Thus, for all $\bw'\in\mathbb B$, we have that \[\nu_B(\mathbf{w'})=\sum_{\w \in \mathbb{B}} \P_B (\mathbf{w},\w')\nu_B(\w ),\]
which means exactly that $\nu_B$ is a stationary measure for the chain $(B_n)_{n\in\N}$.

\bigskip

\noindent
\underline{Step 2}: we now prove that $\nu_B(\mathbb{B})<\infty$.

By Lemma \ref{lem:ADM_B}, we know that 
$\w\in\mathbb B\setminus\{\varepsilon\}$ if and only if $\w$ belongs to a set 
\begin{multline}
b_1\,\A^*_1\,b_2\,\A^*_2\dots b_q\,\A^*_q\\
\isdef \left\{w\in \mathbf V^* : w=b_1w^1b_2w^2\dots b_qw^q;\,w^i\in \A^*_i\mbox{, for all } i \in \llbracket 1,q \rrbracket \right\},\quad q\ge 1,
\label{eq:defwords}
\end{multline}
where $b_1,\ldots, b_q$ are elements of 
$\V$ such that $\{ b_1, \dots, b_q \}\in\mathbb{I}\left(\check G\right)$ and such that for all distinct $i, j$ in $\llbracket 1,q \rrbracket$, 
$b_i\in\S$ implies that $ b_i\neq b_j$, and where we denote $$\A_i\isdef\overline{\E(\{b_1,\dots,b_i\})^c},\quad i\in\llbracket 1,q \rrbracket.$$

Equivalently, by highlighting only the first occurence of each letter of $\V$ appearing in $\w$ and employing a similar notation to (\ref{eq:defwords}) 
we obtain that $\w\in\mathbb B\setminus\{\varepsilon\}$ if and only if $\w$ belongs to some set of the form 
$$\C_{\maI,\sigma}\isdef e_{\sigma(1)}\,\B^*_{\sigma(1)}\,e_{\sigma(2)}\,\B^*_{\sigma(2)}\dots e_{\sigma(|\maI|)}\,\B^*_{\sigma(|\maI|)},$$ 
where $\maI = \left\{e_1,...,e_{|\maI|}\right\}\in\mathbb{I}\left(\check G\right)$, 
$\sigma \in\mathfrak{S}_{|\maI|}$, 
and where we denote \[\B_{\sigma(i)}\isdef\overline{\E(\{e_{\sigma(1)},\dots,e_{\sigma(i)}\})^c}\,\cup\, (\{e_{\sigma(1)},\ldots,e_{\sigma(i)}\}\cap \mathcal{V}_2),\quad i\in\llbracket 1,|\maI| \rrbracket.\]
In view of assertion (1) of Lemma~\ref{lem:nu(A*)}, we have that for all $i\in\llbracket 1,k \rrbracket$, 
\begin{align*}
\nu_B(\B_{\sigma(i)})&=\mu(\E(\{e_{\sigma(1)},\dots,e_{\sigma(i)}\}^c)+\mu(\{e_{\sigma(1)},\dots,e_{\sigma(i)}\}\cap \mathcal{V}_2)\\
&=1-\mu(\E(\{e_{\sigma(1)},\dots,e_{\sigma(i)}\}))+\mu(\{e_{\sigma(1)},\dots,e_{\sigma(i)}\}\cap \mathcal{V}_2).
\end{align*}
Since $\{e_{\sigma(1)},\dots,e_{\sigma(i)}\}\in\mathbb{I}(\check G)$, we have, by definition, that 
$\{e_{\sigma(1)},\dots,e_{\sigma(i)}\}\cap\mathcal{V}_2\in\mathbb{I}(G)$ and since the measure $\mu$ satisfies $\textsc{Ncond}(G)$, it follows that 
\begin{align*}
\mu\left(\{e_{\sigma(1)},\dots,e_{\sigma(i)}\}\cap\mathcal{V}_2\right) &< \mu\left(\maE\left(\{e_{\sigma(1)},\dots,e_{\sigma(i)}\}\cap\mathcal{V}_2\right)\right)\\
&\leq \mu\left(\maE\left(\{e_{\sigma(1)},\dots,e_{\sigma(i)}\}\right)\right)\end{align*}
and thereby, that 
$\nu_B(\mathcal{B}_i)<1$. As a conclusion, applying successively all assertions {of} Lemma~\ref{lem:nu(A*)}, we obtain that for all such $\maI$ and $\sigma$, 
\begin{equation*}
\nu_B(\C_{\maI,\sigma})=\prod\limits_{i=1}^{|\mathcal{I}|} \frac{\mu(e_{\sigma(i)})}{\mu(\mathcal{E}(\{e_{\sigma(1)},\dots,e_{\sigma(i)}\}))-\mu(\{e_{\sigma(1)},\dots,e_{\sigma(i)}\}\cap\maV_2)}.
\end{equation*}

The set $\mathbb B$ is the disjoint union of the sets $\C_{\maI,\sigma}$, for $\maI$ in the finite set $\mathbb I\left(\check G\right)$, and $\sigma$ 
in the finite set $\mathfrak{S}_{|\maI|}$.  
It follows that $\nu_B(\mathbb{B})$ is finite, and given by 
\begin{align}
\nu_B(\mathbb B)&=\nu_B(\varepsilon) + \sum\limits_{\mathcal{I}\in\mathbb{I}\left(\check G\right)}\sum\limits_{\sigma\in\mathfrak{S}_{|\mathcal{I}|}} 
\nu_B(\C_{\maI,\sigma})\nonumber\\
&=1+ \sum\limits_{\mathcal{I}\in\mathbb{I}\left(\check G\right)}\sum\limits_{\sigma\in\mathfrak{S}_{|\mathcal{I}|}}\prod\limits_{i=1}^{|\mathcal{I}|} \frac{\mu(e_{\sigma(i)})}{\mu(\mathcal{E}(\{e_{\sigma(1)},\dots,e_{\sigma(i)}\}))-\mu(\{e_{\sigma(1)},\dots,e_{\sigma(i)}\}\cap\maV_2)}.\label{eq:defalpha0}
\end{align}

\bigskip

\noindent
\underline{Step 3}: we conclude with the positive recurrence of the two chains. 

By the results above, the chain $(B_n)_{n\in\N}$ has a stationary probability distribution on $\mathbb{B}$, which is given by the measure $\nu_B$ normalized by $\nu_B(\mathbb{B})$. 

Observe that the chain is irreducible on $\mathbb{B}$. To see this, let $\bw\in\mathbb B$ and first observe that the empty word $\varepsilon$ leads to $\bw$ with positive probability for the transitions of $\suite{B_n}$ 
(this is the constructive argument proving Lemma \ref{lem:ADM_B} - see the proof of Lemma 1 in \cite{MBM17}). Conversely, denoting by $b_1,\dots,b_q$ the elements of $\maV(\bw)$, it is easy to see that the word $\bw$ leads to the empty word with positive probability for the transitions of $\suite{B_n}$ : 
indeed, by the definition of the policy {\sc fcfm}, if the chain is in the state $\bw$, then it will reach the empty state after exactly $q$ steps, 
by seeing the successive arrivals of $q$ elements of respective classes in $\maE(b_1)$, $\maE(b_2), \dots , \maE(b_q)$, which concludes the proof 
of irreducibility. 

It then follows that the chain $\suite{B_n}$ is positively recurrent on $\mathbb{B}$ and that its stationary probability distribution is unique. Consequently, $\nu_B$ is the unique stationary measure (up to a multiplicative constant) of the chain $(B_n)_{n\in\N}$.

Now, 
as in step 1, we obtain that for all $\w'\in \mathbb B$, 
\begin{equation*}
\nu_F(\Phi(\mathbf{w'}))=\sum\limits_{\w \in \mathbb{B}} \P_F (\Phi(\mathbf{w}),\Phi(\w' ))\nu_F(\Phi(\w )).
\end{equation*}
Using Lemma~\ref{lem:ADMB_ADMF_iso}, we deduce that $\nu_F$ is a stationary measure for the chain $(F_n)_{n\in\N}$. 
Then, step 2 shows equivalently that $\nu_F(\mathbb{F}) <\infty$. 
So, the chain $(F_n)_{n\in\N}$ has a stationary probability distribution on $\mathbb{F}$, which is given by the measure $\nu_F$ normalized by $\nu_F(\mathbb{F})$. 

Similarly as above, we can check that the chain $\suite{F_n}$ is irreducible on $\mathbb{F}$. First, the empty word leads with positive probability 
to any element $\bw\in\mathbb F$, as can be checked using the same constructive argument as in the proof of Lemma 2 
Conversely, suppose that the chain $\suite{F_n}$ is at time $n$ in a state 
\[\bw = a_{qk_q} \dots a_{q1}\td{a_q}\dots\td{a_3} a_{2k_2}\dots a_{21}\td{a_2}a_{1k_1}\dots a_{11} \td{a_1}\in\mathbb F\]
and let $r\isdef q+\sum_{i=1}^q k_i$ be the length of $\bw$. Then, going forward in time, perform the {\sc fcfm}-matching of 
the `unmatched' elements of respective classes in $\maV(\bw)$. Say there remains in the system, at time $n+r$, $\ell$ unmatched elements denoted $c_1,c_2,\dots,c_\ell$ in their order of arrivals. Then, the chain can return to the empty state in particular if the first $\ell$ arrivals 
after time $n+r$ (excluded) are of respective classes in $\maE(c_1)$, $\maE(c_2),\dots ,\maE(c_\ell)$. This concludes the proof of irreducibility. 

As a consequence, the chain $\suite{F_n}$  is positively recurrent on $\mathbb{F}$ and its stationary probability distribution is unique. 
Consequently, $\nu_F$ is the unique stationary measure (up to a multiplicative constant) of the chain $(F_n)_{n\in\N}$, which concludes the proof. 
\end{proof}

\subsection{Positive recurrence of $\suite{W_n}$.} 
The Markov chain $(W_n)_{n\in\N}$ can be seen as the projection of the chain $(B_n)_{n\in\N}$ on $\V^*$. In order to obtain the stationary probability distribution of $(W_n)_{n\in\N}$ from the one of $(B_n)_{n\in\N}$, we will use the following lemma:
\begin{lemma}\label{lem:projCM} Let $\P_Y$ and $\P_{Y'}$ be the transition matrices of two homogeneous 
Markov chains $\suite{Y_n}$ and $\suite{Y'_n}$ with values in some countable sets $S$ and $S'$ respectively, and consider a map 
$p:S \to S'$ satisfying 
\begin{equation*}
\forall a',b'\in S', \; \forall a\in p^{-1}(\{a'\}), \P_Y(a,p^{-1}(\{b'\})) = \P_{Y'}(a',b').
\end{equation*} 
Then, if a measure $\mu$ is invariant for $\P_Y$, the measure $\mu'$ defined by $\mu'(a')\isdef \mu(p^{-1}(\{a'\}))$ for all $a'\in S'$, 
is an invariant measure for $\P_{Y'}$ on $S'$.
\end{lemma}

\begin{proof} Let $\mu$ be an invariant measure for $\P_Y$, and let $b'\in S'$. We have
\begin{align*}
\sum_{a'\in S'} \mu'(a')\P_{Y'}(a',b')
&= \sum_{a'\in S'} \left(\sum_{a\in p^{-1}(\{a'\})} \mu(a)\right)\P_{Y'}(a',b') \\
&= \sum_{a'\in S'} \sum_{a\in p^{-1}(\{a'\})} \mu(a) \P_Y(a,p^{-1}(\{b'\}))\\
&= \sum_{s\in S} \mu(s) \P_Y(s,p^{-1}(\{b'\}))=\mu(p^{-1} (\{b'\}))=\mu'(b'),
\end{align*}
meaning that $\mu'$ is invariant for $\P_{Y'}$.
\end{proof}

For $\mu\in\textsc{Ncond}(G)$, let us denote by $\Pi_B$ the unique stationary probability law associed to the chain $(B_n)_{n\in\N}$ (cf. Prop.~\ref{prop:B_nF_n_rec_pos}). It is defined by
\begin{equation*}
\forall \w \in \mathbb{B}, \; \Pi_B(\w )\isdef \alpha \nu_B(\w), 
\end{equation*}
where $\alpha\isdef(\nu_B(\mathbb B))^{-1}$ is given by (\ref{eq:defalpha}) in view of (\ref{eq:defalpha0}). 
Let us now introduce the projection 
\[p:\begin{cases}
\mathbb{B} \longrightarrow \mathbb{W}\\
 \bw \longmapsto \bw|_{\V},
 \end{cases}
 \]
 which is well-defined from Lemma~\ref{lem:ADM_B}. We have the following result: 

\begin{proposition}\label{prop:W_nRec}
Let $\mu\in\textsc{Ncond}(G)$. Then, the Markov chain $(W_n)_{n\in\N}$ is positively recurrent, and its unique stationary probability distribution is the measure $\Pi_W$ defined on $\mathbb W$ by:
\begin{equation*}
\forall w\in\mathbb{W}, \; \Pi_W(w)\isdef \Pi_B(p^{-1}(w))=\sum_{\w \in \mathbb{B} \; : \; \w |_{\V}=w} \Pi_B(\w). 
\end{equation*}
\end{proposition}

\begin{proof} Let $\mu\in\textsc{Ncond}(G)$. We can apply Lemma~\ref{lem:projCM} to $\P_B$ and $\P_W$ to prove that $\Pi_W$ is a stationary distribution for $(W_n)_{n\in\N}$. 
Indeed, using the fact that for any $n\in\mathbb{N}$, $W_n$ is the restriction of the word $B_n$ to its letters in $\V$, we have that 
\begin{equation*}
\forall w,w'\in \mathbb{W}, \; \forall \w\in p^{-1}(\{w\}), \; \P_B(\w,p^{-1}(\{w'\})) = \P_W(w,w').
\end{equation*} 
The measure $\Pi_W$ is a probability distribution on $\mathbb{W}$, since $\Pi_W(\mathbb{W})=\Pi_B(p^{-1}(\mathbb{W}))=\Pi_B(\mathbb{B})=1$.
The chain $(W_n)_{n\in\N}$ being irreducible on $\mathbb{W}$, it follows that $\Pi_W$ is its unique stationary probability distribution.
\end{proof}

\subsection{Concluding the proof.}
We first show that $\textsc{stab}(G,\fcfm)=\textsc{Ncond} (G)$. 
First, we know from Proposition \ref{thm:mainmono} that $\textsc{stab}(G,\fcfm) \subset \textsc{Ncond}(G).$ Also, 
from Proposition \ref{prop:ncond}, $ \textsc{Ncond}(G)\neq \emptyset$, since $G$ is not a bipartite graph. 
Then, for all $\mu \in \textsc{Ncond}(G)$, by Proposition~\ref{prop:W_nRec}, the chain $(W_n)_{n\in\N}$ is positively recurrent on $\mathbb{W}$. 
So, $\textsc{Ncond}(G)\subset \textsc{stab}(G,\fcfm)$, and therefore $\textsc{stab}(G,\fcfm)=\textsc{Ncond} (G)$. 

\bigskip

We now fix $\mu\in \textsc{Ncond}(G)$, and compute explicitly the unique stationary probability distribution $\Pi_W$ of the chain 
$(W_n)_{n\in\N}$. 
First, if $w=\varepsilon$, then $p^{-1}(\{w\})=\varepsilon$ and $\Pi_W(\varepsilon)=\alpha$, given by (\ref{eq:defalpha}). 
Now, fix $w\neq\varepsilon$ in $\mathbb W$. 
By (\ref{eq:defwords}), we know that if $w=w_1\dots w_q\in \mathbb{W}$, $q\ge 1$, then
$p^{-1}(\{w\})=w_1\,\A^*_1\,w_2\,\A^*_2\dots w_q\,\A^*_q,$ with $\A_i=\overline{\E(\{w_1,\dots,w_i\})^c}$, for all $i\in\llbracket 1,q \rrbracket$. 
Applying Lemma \ref{lem:nu(A*)} and observing that 
for all $i$, $\mu\left(\overline{\A_i}\right)<1$ since $\A_i\varsubsetneq\oV$, it follows that 
\begin{align*}
\Pi_W(w)&=\Pi_B(w_1\,\A^*_1\,w_2\,\A^*_2\dots w_q\,\A^*_q)\\
&=\alpha \nu_B(w_1\,\A^*_1\,w_2\,\A^*_2\dots w_q\,\A^*_q)\\
&=\alpha \prod_{i=1}^q{\mu(w_i)\over 1-\mu\left(\overline{\A_i}\right)} 
=\alpha \prod_{i=1}^q{\mu(w_i)\over \mu(\E(\{w_1,\dots,w_i\}))}\cdot
\end{align*}
The proof is complete. 


\subsection{Observation on the sequence of copied letters.}
\label{subsec:iid}
As we will explain below, observe that under the {\sc fcfm} policy, the i.i.d. sequence of arrivals induces an i.i.d. sequence for the copied letters, with the same distribution on $\overline \maV$ as on $\maV$. Notice that a similar result was observed for the related construction in \cite{ABMW17}. 

Let ${\mathbb M}\subset {\maV}^*$ represent the set of words $z\in\maV^*\setminus\{\varepsilon\}$ such that $W^{\tiny{\textsc{fcfm}}}(z)=\varepsilon$ and $W^{\tiny{\textsc{fcfm}}}(z')\ne \varepsilon$ for any prefix $z'$ of $z$. In other words, a word $z\in\mathbb M$ represents a sequence of arrivals leading to a first reset to zero of the buffer. 
The i.i.d. sequence of letters of distribution $\mu$ can be uniquely decomposed into an i.i.d. sequence of words of ${\mathbb M}$ of distribution $\mu_{{\mathbb M}}$, where for any $w=w_1\ldots w_q\in {\mathbb M}$, $\mu_{{\mathbb M}}(w)=\prod_{i=1}^q\mu(w_i)$. Note also that the distribution $\mu_{{\mathbb M}}$ characterizes $\mu$. In particular, if we have an i.i.d. sequence of words of ${\mathbb M}$ of law $\mu_{{\mathbb M}}$, then the sequence of letters out of which the words are made is an i.i.d. sequence of law $\mu$.

Let us denote by $f:{\mathbb M}\to{\mathbb M}$ the map that, to a word $w=w_1\ldots w_q\in {\mathbb M}$, associates the word $v=v_1\ldots v_q$, where $v_i$ is the class of the item with which $w_i$ is matched. If follows from Prop.~3 of~\cite{MBM17} that the map $f$ is well-defined and bijective, its inverse function being the function $g:{\mathbb M}\to{\mathbb M}, w\mapsto \overleftarrow{f(\overleftarrow{w})}$. Furthermore, by definition of ${\mathbb M}$, the action of $f$ only consists in permuting letters. The measure $\mu_{\mathbb M}$ having a product structure, it follows that $f$ preserves the distribution $\mu_{{\mathbb M}}$: for any $w\in {\mathbb M}, \mu_{\mathbb M}(w)=\mu_{\mathbb M}(f^{-1}(\{w\}))$. Consequently, if $(W_i)_{i\in \N}$ is an i.i.d. sequence of words distributed according to $\mu_{{\mathbb M}}$, then the sequence $(f(W_i))_{i\in\N}$ is also an i.i.d. sequence of words distributed according to $\mu_{{\mathbb M}}$, meaning that the sequence of letters out of which the words $f(W_i)$ are made is an i.i.d. sequence of law $\mu$. 

Thus, replacing each arrival by the class of the item with which it is matched preserves the measure $\mu^{\otimes \N}$. 

\section{Remaining proofs}\label{sec:otherproofs}

Throughout the section, $G$ is a connected multigraph, $\check G$ is its maximal subgraph and $\hat G$ denotes its minimal blow-up graph. 
To simply compare a $(G,\Phi,\mu)$ system with the two corresponding matching models on graphs 
$\left(\hat G,\hat \Phi,\hat\mu\right)$ and $\left(\check G,\check \Phi,\mu\right)$, let us add an ``hat'' (resp. a ``check'') 
to all characteristics of the second (resp. the third) system: in particular, we denote, for all $n$, by $\hat V_n$ (resp. $\check V_n$), the class of the item entering in the 
$\left(\hat G,\hat\Phi,\hat\mu\right)$ (resp. $\left(\check G,\check\Phi,\mu\right)$) system at time $n$. 
The natural Markov chain of the system is then denoted by $\suite{\hat W_n}$ (resp. $\suite{\check W_n}$) 
and its state space, by $\hat{\mbW}$ (resp. $\check{\mbW}$). Specifically, 
\begin{align*}
\hat{\mathbb W} &=\Bigl\{ w\in \left(\maV\cup \underline{\maV_1}\right)^*\; : \; \forall  i\neq j \; \text{s.t.} \; (i,j) \in \hat\maE, \; |w|_i|w|_j=0 \Bigr\};\\
\check{\mathbb W} &=\Bigl\{ w\in \maV^*\; : \; \forall  i\neq j \; \text{s.t.} \; (i,j) \in \check\maE, \; |w|_i|w|_j=0 \Bigr\}.
\end{align*} 
Observe that we have $\mathbb W \subset \check{\mathbb W} \subset \hat{\mathbb W}$. 

For any measurable mapping $F:\mbW \to \R$ (resp. $\check{\mbW} \to \R$, $\hat{\mbW} \to \R$) and any given $w\in\mbX$ (resp. $\hat{w}\in\hat{\mbW}$, $\check{w}\in\check{\mbW}$), we denote by 
$\Delta F^\Phi_\mu(w)$ (resp. $\hat{\Delta} F^{\hat\Phi}_{\hat\mu}(\hat w)$, $\check{\Delta} F^{\check\Phi}_{\check\mu}(\check w)$) the drift of the chain $\suite{W_n}$ (resp. $\suite{\hat{W}_n}$, 
$\suite{\check{W}_n}$) starting from $w$ (resp. $\hat{w}$, $\check{w}$) for a $(G,\Phi,\mu)$  (resp. $(\hat G,\hat \Phi,\hat \mu)$, $(\check G,\check\Phi,\check\mu)$) system. 
In other words, for any $n\in\N$ we denote 
\begin{align*}
\Delta^\Phi_\mu F(w)&=\esp{F(W_{n+1})-F(W_n)\,\big|\,W_n=w} ;\\
\hat{\Delta}^{\hat\Phi}_{\hat\mu}F(\hat w)&=\esp{F(\hat W_{n+1})-F(\hat W_{n})\,\big|\,\hat W_n=\hat w};\\ 
\check{\Delta}^{\check\Phi}_{\check\mu}F(\check w)&=\esp{F(\check W_{n+1})-F(\check W_{n})\,\big|\,\check W_n=\check w}. 
\end{align*}

\subsection{Drift inequalities}
\label{sec:drift}


Consider the following  mappings, 
\begin{equation}
\label{eq:QuadraticFunction}
Q:\begin{cases}
\hat{\mathbb W} &\longrightarrow\R^+\\
\hat w &\longmapsto \sum\limits_{i=1}^{|\maV|}\left(|\hat w|_i\right)^2 + \sum\limits_{i=1}^{|\maV_1|}\left(|\hat w|_{\underline i}\right)^2;
\end{cases}\end{equation}
\begin{equation}
\label{eq:LinearFunction}
L: \begin{cases}
\hat{\mathbb W} &\longrightarrow \R^+\\
\hat w &\longmapsto \sum\limits_{i=1}^{|\maV|} |\hat w|_i + \sum\limits_{i=1}^{|\maV_1|}|\hat w|_{\underline i};\quad\;\;\quad 
\end{cases}
\end{equation}
where it follows from the observation above that $Q$ and $L$ are well defined also on $\mathbb W$ and $\check{\mathbb W}$. 

\medskip


\noindent We have the following result,

\begin{proposition}
	\label{prop:extquad}
	Let $\Phi$ be an admissible policy on $G$ and $\mu\in\mathscr M(\maV)$. 
	Let $\hat{\Phi}$ be a matching policy extending $\Phi$ on $\hat G$ and $\hat \mu$ be a measure extending $\mu$ on $\hat\maV$. 
	Then, 
	for all $w \in\mbW$ we have that 
$\Delta^\Phi_\mu Q(w)
\leq 
\hat\Delta^{\hat\Phi}_{\hat\mu} Q(w). $
\end{proposition}

\begin{proof}
Fix $w \in\mbW$ throughout the proof. Recall that for all $i\in\maV_1$ (if any), we have that $|w|_i\in\{0,1\}$, and let us set  
\begin{align*}
\mathcal O_{w} &\isdef\{i\in \maV_1:|w|_i = 1\},\\
\mathcal Z_{w} &\isdef\{i\in \maV_1: |w|_i = 0\mbox{ and } |w|_j = 0\mbox{, for any }j \in \maE(i)\}.
\end{align*}


 \medskip 
 
{First, for any $i \in\maV_2$, an incoming item of class $i$ finding a system $(G,\Phi,\mu)$ in a state $w$ finds the same possible matches (if any) as an incoming item of class $i$ finding the system $\left(\hat G,\hat \Phi,\hat \mu\right)$ in a state $w$. Likewise, if a system $(G,\Phi,\mu)$ is in state $w$, then, for any $i \in \mathcal V_1\cap (\mathcal O_{w})^c \cap (\mathcal Z_{w})^c$, an incoming of class $i$ 
finds the same possible matches (if any) as an incoming item of class $i$ or of class $\underline{i}$ finding the system $\left(\hat G,\hat \Phi,\hat \mu\right)$ in the state $w$. 
As $\hat\Phi$ extends $\Phi$, in all cases the choice of the match (if any) of the incoming item is then the same, or follows the same distribution in case of a draw, in both systems. 
Therefore, as $\hat\mu$ extends $\mu$, for all $n\in\N$ we get that 
%
%
%
\begin{multline}\label{eq:compML2}
\esp{\left(Q(W_{n+1})-Q(W_n)\right)\ind_{\{V_{n+1} \in \mathcal V_2\cup\mathcal V_1\cap (\mathcal O_{w})^c \cap (\mathcal Z_{w})^c\}}\,\big|W_n= w}\\
= \mathbb E\Biggl[\!\left(\!Q\!\left(\hat W_{n+1}\right)\!\!-\!Q\!\left(\hat W_{n}\right)\!\right)\!\!\ind_{\left\{\hat V_{n+1} \in \mathcal V_2\cup\left( \mathcal V_1\cap (\mathcal O_{w})^c \cap (\mathcal Z_{w})^c\right) \cup \left(\underline{\mathcal V_1\cap (\mathcal O_{w})^c \cap (\mathcal Z_{w})^c}\right)\right\}}\!\big|\hat W_n=w\!\Biggl].
\end{multline}
Now, if a system $(G,\Phi,\mu)$ is in state $w$, then, for any $i\in \mathcal Z_{w}$, and incoming item of class $i$ finds no possible match. All the same, 
 if the system $\left(\hat G,\hat\Phi,\hat\mu\right)$ is in the state $w$ and the entering item is of class $i\in \mathcal Z_{w}$ or of class $\underline i \in \underline{\mathcal Z_{w}}$, then the entering item does not find any possible match in $\left(\hat G,\hat\Phi,\hat\mu\right)$. In all cases, one coordinate of the Markov chain increases from zero to 
 one, and thus for all $n\in\N$, 
\begin{multline}\label{eq:compML3}
\esp{\left(Q(W_{n+1})-Q(W_n)\right)\ind_{\{V_{n+1} \in\mathcal Z_{w}\}}\,|\,W_n= w}\\
\begin{aligned}
=& \;\mu(\maZ_{w})\\
=& \;\hat\mu\left(\maZ_{w}\right) + \hat\mu\left(\underline{\maZ_{w}}\right)\\
=& \;\esp{\left(Q(\hat W_{n+1})-Q(\hat W_{n})\right)\ind_{\{\hat V_{n+1} \in \mathcal Z_{w}\cup \underline{\maZ_{w}}\}}\,|\,\hat W_n=w}.
\end{aligned}\end{multline}}

\medskip
\noindent


Last, if the system $\left(G,\Phi,\mu\right)$ is in the state $w$, then, the arrival of a class $i$-item, for $i\in \mathcal O_{w}$, leads to the matching of two items 
of class $i$. Therefore, as $|w|_i=1$ we obtain  
\begin{equation}
\left(Q(W_{n+1})-Q(W_n)\right)\ind_{\{V_{n+1} \in \mathcal O_{w}\}}
= -\ind_{\{V_{n+1} \in \mathcal O_{w}\}}.\label{eq:compML4bis}
\end{equation}
Now, suppose that the system $\left(\hat G,\hat\Phi,\hat\mu\right)$ is in the state $\hat W_n =w$. 
Then, if an item of class $i\in \mathcal O_{w}$ enters in the system, the corresponding item is not matched and the number of $i$-items in the system increases from 1 to 2. 
Therefore we get that 
 \begin{equation}
 \label{eq:compML4ter}
 \left(Q\left(\hat W_{n+1}\right)-Q\left(\hat W_{n}\right)\right)\ind_{\left\{\hat V_{n+1} \in \mathcal O_{w}\right\}}
= 3\ind_{\left\{\hat V_{n+1} \in \mathcal O_{w}\right\}}.
 \end{equation}
  \noindent
If on the other hand, an item of class $\underline i\in \underline{\mathcal O_{w}}$ enters in the same system $\left(\hat G,\hat\Phi,\hat\mu\right)$, 
then, the corresponding item match with the stored class $i$-item and so the coordinate $i$ of the chain decreases to 0. Thus, 
 \begin{equation*}
\left(Q\left(\hat W_{n+1}\right)-Q\left(\hat W_{n}\right)\right)\ind_{\left\{\hat V_{n+1} \in \underline{\mathcal O_{w}}\right\}}
= -\ind_{\left\{\hat V_{n+1} \in \underline{\mathcal O_{w}}\right\}}.
  \end{equation*}
Gathering this with (\ref{eq:compML4bis}) and (\ref{eq:compML4ter}) and then taking expectations, we obtain that 
\begin{multline}\label{eq:compML4}
\esp{\left(Q(W_{n+1})-Q(W_n)\right)\ind_{\{V_{n+1} \in\mathcal O_{w}\}}\,\big|\,W_n= w}\\
= \esp{\left(Q\left(\hat W_{n+1}\right)-Q\left(\hat W_{n}\right)\right)\ind_{\left\{\hat V_{n+1} \in \mathcal O_{w}\cup \underline{\mathcal O_{w}}\right\}}\,\big|\,\hat W_n=w} -4\hat\mu\left(\mathcal O_{w}\right).
\end{multline}
Finally, (\ref{eq:compML2}) together with (\ref{eq:compML3}) and (\ref{eq:compML4}) give that
\begin{equation} \label{eq:compMLfinal}
\Delta^\Phi_\mu Q(w)
=\hat\Delta^{\hat\Phi}_{\hat\mu} Q(w)- 4\hat\mu\left(\mathcal O_{w}\right),
\end{equation}
which concludes the proof. 
\end{proof}

\begin{definition}
	\label{def:reducesspol}
	Let $G$ be a connected multigraph and $\Phi$ be an admissible matching policy on $G$. We say that $\check\Phi$ reduces $\Phi$ if, 
	for any ${\mu}\in\mathscr M({{\maV}})$, whenever the two systems  $(\check G,\check\Phi,\mu)$ and $(G,\Phi,\mu)$ are in the 
	same state ${w} \in {\mbW}$ and welcome the same arrival, then $\check\Phi$ and $\Phi$ induce the same choice of match, if any. 
\end{definition}



\begin{proposition}
	\label{prop:EspGmultiLeqEspGgrapg}
	Let $G=(\maV,\maE)$ be a connected multigraph and $\Phi$ be a class admissible policy on $G$ and $\mu\in\mathscr M(\maV)$. 
	Let $\hat{\Phi}$ be a matching policy extending $\Phi$ on $\hat G$, $\hat \mu$ a measure extending $\mu$ on $\hat\maV$ 
	and $\check \Phi$ be a policy that reduces $\Phi$ on $\check G$. 
	Then the drift of the respective Markov chains are such that for all $w\in\mbW$, 
	\begin{equation}
	\label{eq:driftL}
	{\Delta}^{\Phi}_{\mu}L(w) \le \hat{\Delta}^{\hat\Phi}_{\hat\mu}L(w) \le \check{\Delta}^{\check\Phi}_{\mu}L(w).
	\end{equation}
	\end{proposition}
\begin{proof}
	Fix $w\in\mbX$. 
	The only case in which the proof of the left inequality of (\ref{eq:driftL}) differs from that of Proposition \ref{prop:extquad} is when 
	an item of class $i\in \mathcal O_{w}$ enters the $(\hat G,\hat\Phi,\hat\mu)$ system in a state $w$. 
	Then, we now get that for all $n$, 
	\begin{equation*}
	\left(L(\hat W_{n+1})-L(\hat W_{n})\right)\ind_{\{\hat V_{n+1} \in \mathcal O_{w}\}}
	=\sum_{i\in \mathcal O_{w}}\ind_{\{\hat V_{n+1} =i\}},
	\end{equation*}	
	\noindent which, taking expectations and reasoning as in (\ref{eq:compMLfinal}), leads to 
	\begin{equation*} 
	{\Delta}^{\Phi}_{\mu}L(w) 
	= \hat{\Delta}^{\hat\Phi}_{\hat\mu}L(w)  
	- 2\hat\mu\left(\mathcal O_{w}\right).
	\end{equation*} 
	\medskip
	We now turn to the proof of the right inequality of (\ref{eq:driftL}). Denote 
	\begin{equation*}
	\mathcal P_{w} =\{i\in \maV_1\,:\,|w|_i > 0\}. 
	\end{equation*}
	Fix also $n\in\N$, and denote by $\check V_{n}$, the class of the incoming item at time $n$ 
	in the $(\check G,\check\Phi,\mu)$ system. First, similarly to (\ref{eq:compML2}) and (\ref{eq:compML3}) we clearly get that 
	\begin{multline}\label{eq:compsep1}
	\esp{\left(L(\check W_{n+1})-L(\check W_{n})\right)\ind_{\{\check V_{n+1} \in \maV_2 \cup \left(\mathcal V_1\cap (\mathcal P_{w})^c \right)\}}
		\,|\,\check W_n= w}\\
	=\esp{\left(L(\hat W_{n+1})-L(\hat W_{n})\right)\ind_{\left\{\hat V_{n+1} \in \maV_2\cup \left(\mathcal V_1\cap (\mathcal P_{w})^c \cup \left(\underline{\mathcal V_1\cap (\mathcal P_{w})^c }\right)\right)\right\}}|\hat W_n=w}.
	\end{multline}
	Now, we also clearly have that 
	\begin{align*}
	\esp{\left(L(\check W_{n+1})-L(\check W_{n})\right)\ind_{\{\check V_{n+1} \in \mathcal P_{w}\}}
		\,|\,\check W_n= w} &=  \mu(\mathcal P_w);\\
	\esp{\left(L(\hat W_{n+1})-L(\hat W_{n})\right)\ind_{\{\hat V_{n+1} \in \mathcal P_{w}\}}
		\,|\,\hat W_n= w} &=  \hat \mu(\mathcal P_{w});\\
	\esp{\left(L(\hat W_{n+1})-L(\hat W_{n})\right)\ind_{\{\hat V_{n+1} \in \underline{\mathcal P_{w}}\}}
		\,|\,\hat W_n= w} &= -\hat \mu(\underline{\mathcal P_{w}}),
	\end{align*}
	which, together with (\ref{eq:compsep1}), implies that 
	\begin{equation*} 
	 \hat{\Delta}^{\hat\Phi}_{\hat\mu}L(w) 
	= \check{\Delta}^{\check\Phi}_{\mu}L(w)
	-2\hat \mu(\underline{\mathcal P_{w}}).
	\end{equation*}
\end{proof}


\subsection{Proofs of the remaining main results}
\label{subsec:proofs}

We are now in position to prove Theorem \ref{thm:ML}, Theorem \ref{thm:ppartite} and Proposition \ref{prop:extppartite}. 
Let us first observe the following result, 
\begin{lemma}
	\label{lemma:equivalenceNcond}
	For any $\mu\in\mathcal M(\maV)$ we have that
	\begin{equation*}
\mu\in\textsc{Ncond}(G)\iff\hat{\mu}_{\tiny{1/2}}\in\textsc{Ncond}\left(\hat{G}\right), 
	\end{equation*}
	where $\hat{\mu}_{\tiny{1/2}}$ is the extended measure of $\mu$ such that 
	$$\displaystyle\hat{\mu}_{\tiny{1/2}}(i) = \hat{\mu}_{\scriptsize{1/2}}(\underline{i}) ={1 \over 2} \mu(i)\quad \text{for all} \; i\in\maV_1.$$ 
	\end{lemma}
\begin{proof}
	Let us first observe that 
	\begin{equation}
	\label{eq:result}
	\mbox{For any }A\subset \maV,\quad \mu(\maE(A))=\hat{\mu}_{\tiny{1/2}}\left(\hat\maE(A)\right).
	\end{equation}
	Indeed we have that 
	\begin{align*}
	 {\mu}(\maE({A}))
&={\mu}(\maE({A})\cap \maV_2)+{\mu}(\maE({A})\cap \maV_1)\\
&=\hat{\mu}_{\tiny{1/2}}(\maE({A})\cap \maV_2) + \hat{\mu}_{\tiny{1/2}}(\maE({A})\cap \maV_1)
+ \hat{\mu}_{\tiny{1/2}}\left(\cop{\maE({A})\cap \maV_1}\right)\\
&=\hat{\mu}_{\tiny{1/2}}\left((\maE({A})\cap \maV_2) \cup (\maE({A})\cap \maV_1) \cup 
\left(\cop{\maE({A})\cap \maV_1}\right)\right) = \hat{\mu}_{\tiny{1/2}}\left(\hat{\maE}\left({A}\right)\right). 
\end{align*}
	\noindent $\Leftarrow$: Let $\hat{\mu}_{\tiny{1/2}}\in\textsc{Ncond}\left(\hat{G}\right)$ and $\mathcal{I}\in\mathbb{I}(G)$. As $\maI \subset \maV_2$ and in view of 
	(\ref{eq:result}), we get that 
\begin{equation*}
	\mu(\mathcal{I})=\hat{\mu}_{\tiny{1/2}}(\mathcal{I})<\hat{\mu}_{\tiny{1/2}}\left(\hat{\maE}(\mathcal{I})\right)=\mu(\maE(\mathcal{I})).
	\end{equation*}
	\noindent $\Rightarrow$: Let us now fix $\mu\in\textsc{Ncond}(G)$ and $\hat{\mathcal{I}} \in \I\left(\hat G\right)$. 
We reason by induction on the number of elements of $\hat{\mathcal{I}}$ that belong to $\maV_1\cup \cop{\maV_1}$.  
First, observe that if $\hat{\mathcal{I}}\subset \maV_2$, then {$\hat{\maI}$ is also an element of $\I(G)$. Thus, as 
$\mu\in\textsc{Ncond}(G)$ and in view of (\ref{eq:result}) we get that 
\begin{align*}
\hat{\mu}_{\tiny{1/2}}(\hat{\mathcal{I}})={\mu}(\hat{\mathcal{I}}) < {\mu}(\maE(\hat{\mathcal{I}}))
= \hat{\mu}_{\tiny{1/2}}\left(\hat{\maE}\left(\hat{\maI}\right)\right). 
\end{align*}
}
 Second, assume that for a given $\hat{\mathcal{I}}\in \I\left(\hat G\right)$, we have $\hat{\mu}_{\tiny{1/2}}(\hat{\mathcal{I}})<\hat{\mu}_{\tiny{1/2}}(\maE(\hat{\mathcal{I}})),$ and let $u\in \maV_1\cup \cop{\maV_1}$ be such that $u\not\in\hat{\mathcal{I}}$ and $\hat{\mathcal{I}}\cup\{u\}\in \I\left(\hat G\right)$. Then, 
we have $u\notin \hat\maE(\hat{\mathcal{I}})$, and thus also $\cop{u}\notin \hat\maE(\hat{\mathcal{I}})$. 
Since $\cop{u}\in \hat\maE(\hat{\mathcal{I}}\cup\{u\})$, it follows that 
{\begin{align*}
\hat{\mu}_{\tiny{1/2}}(\hat\maE(\hat{\mathcal{I}}\cup\{u\}))\ge \hat{\mu}_{\tiny{1/2}}(\hat\maE(\hat{\mathcal{I}}))+\hat{\mu}_{\tiny{1/2}}(\cop{u})&=\hat{\mu}_{\tiny{1/2}}(\hat{\maE}(\hat{\mathcal{I}}))+\hat{\mu}_{\tiny{1/2}}(u)\\
&>\hat{\mu}_{\tiny{1/2}}(\hat{\mathcal{I}})+\hat{\mu}_{\tiny{1/2}}(u)
=\hat{\mu}_{\tiny{1/2}}(\hat{\mathcal{I}}\cup\{u\}),
\end{align*}}
and we conclude by induction. 
\end{proof}

\begin{proof}[Proof of Theorem \ref{thm:ML}]
	Let $\mu\in \textsc{Ncond}(G)$. From Lemma \ref{lemma:equivalenceNcond}, the measure $\hat\mu_{\tiny{1/2}}$ belongs to {\sc Ncond}$\left(\hat G\right)$. 
	Let $\Phi$ be a matching policy of the Max-Weight class on $G$, with $\beta >0$. Clearly, its extension $\hat\Phi$ is also 
	of the Max-Weight class on $\hat G$. Then, we know from Theorem 3.2 in \cite{JMRS20} that the model $\left(\hat G,\hat\Phi,\hat\mu_{\tiny{1/2}}\right)$ is stable. 
	In particular, 
	we see in the proof of Theorem 3.2 in \cite{JMRS20} that 
	the Foster-Lyapunov Theorem (\cite{Bre99}, Theorem 5.1) can be applied to the chain $\left(\hat W_n\right)_{n\in\mathbb{N}}$ for the quadratic function $Q$. 
	Specifically, there exist $\eta >0$ and a finite set $\hat{\mathcal { K}} \subset \hat{\mbW}$ such that 
	$\hat\Delta^{\hat\Phi}_{\hat\mu} Q(\hat{w}) <-\eta$ for all $\hat{w} \not\in \hat{\mathcal K}$. 
	Thus, in view of Proposition \ref{prop:extquad}, we have that $\Delta^{\Phi}_{\mu} Q({w}) <-\eta$ 
	for any $w$ that lies outside the finite subset $\mathcal K=\hat{\mathcal K}\cap \mbW$. 
	We conclude by applying the Foster-Lyapunov Theorem to the mapping $Q$ and the compact set $\mathcal K$. 
	\end{proof}

\begin{proof}[Proof of Theorem \ref{thm:ppartite}]
(i)  Fix $\mu\in \textsc{Ncond}(G)$. 
		First, if $G$ is a graph ($\maV_1 =\emptyset$) and $G=\check G$ is $p$-partite complete for $p\ge 3$, then 
		the result follows from Theorem 2, Assertion (16) in \cite{MaiMoy16}: specifically, we have that for some $\eta>0$, 
		for any $w\in\mbW\setminus\{\varepsilon\},$ 
		\begin{equation}
		\label{eq:majoreG}
		{\Delta}^{\Phi}_{\mu}L(w)
		< -\eta,
		\end{equation}
		and the Foster-Lyapunov criterion applies. 
		Now, if $G$ is not a graph, i.e. $\maV_1\ne\emptyset$, then let 		
		\[\delta=\min\left\{\mu(\maE(\maI))-\mu(\maI)\,:\,\maI\in\I(G)\right\},\]
		which is strictly positive since $\mu\in \textsc{Ncond}(G)$, and the mapping 
		\[L_\delta : \left\{\begin{array}{ll}
		\mbW &\longrightarrow \R+\\
		w&\longmapsto \displaystyle\sum_{i\in\maV_1} {\delta \over 2\mu(\maV_1)} |w|_i + \sum_{i\in\maV_2} |w|_i.  
		\end{array}\right.
		\]
		Then, for any $w\in\mbW$ the set $\maI^{w}=\{i\in\maV\,:\,|w|_i>0\}$ is an independent set of $\check G$, so by the very definition of a complete $p$-partite graph, there exists a unique maximal independent set 
		$\check{\maI}$ of $\check G$ such that $\maI^{w} \subset \check{\maI}$. 
		Then, for all $w\in\mbW$ such that $\maI^{w}\cap\maV_2\ne \emptyset$, for any $n$, if $\{W_n=w\}$ the Markov chain can make two types of moves upon the arrival of $V_{n+1}$: 
		\begin{itemize}
			\item either one coordinate of $W_n$ decreases from 1 if $V_{n+1}$ is of a class in $\check{\maI}^c=\check{\maE}\left(\check{\maI}\right)$, or of a class in $\maI^{w}\cap \maV_1$;
			\item or one coordinate of $W_n$ increases from 1, if $V_{n+1}$ is of a class in $\check{\maI}\cap\left((\maI^{w})^c\cup \maV_2\right)$. 
		\end{itemize} 
		Therefore, for any $\maV_2$-favorable matching policy $\Phi$ we have that 
		\begin{multline} \label{eq:vinci1}
		{\Delta}^{\Phi}_{\mu}L_\delta(w)\\
		=-{\delta \over 2\mu(\maV_1)} \mu\left(\maV_1 \cap \maI^{w}\right)\ind_{\{\maV_1 \cap \maI^{w}\ne \emptyset\}}+ {\delta \over 2\mu(\maV_1)}\mu\left(\maV_1 \cap \check{\maI}\cap(\maI^{w})^c\right)
		+\mu\left(\check{\maI}\cap \maV_2\right) - \mu\left(\check{\maI}^c\right).
		\end{multline} 
	  Observe that $\check{\maI}\cap \maV_2$ is an independent set of $G$, and that $\check{\maI}^c=\maE\left(\check{\maI}\cap \maV_2\right)$. Hence (\ref{eq:vinci1}) implies that 
		\begin{equation*} 
		{\Delta}^{\Phi}_{\mu}L_\delta(w)
		\le  {\delta \over 2\mu(\maV_1)}\mu\left(\maV_1 \cap \check{\maI}\cap(\maI^{w})^c\right)
		+\mu\left(\check{\maI}\cap \maV_2\right) - \mu\left(\maE\left(\check{\maI}\cap \maV_2\right)\right)
		\le  {\delta \over 2} - \delta =-{\delta \over 2}.
		\end{equation*} 
		As this is true for any $w$ outside the finite set $\{w\in\mbW\,:\, \maI^{w}\cap\maV_2 =\emptyset\}$, we conclude again using the Lyapunov-Foster Theorem that $\textsc{stab}(G,\Phi)=\textsc{Ncond}(G).$

		 (ii) Fix $\mu\in\textsc{Ncond}(\check G),$ and an admissible matching policy $\Phi$. Applying (\ref{eq:majoreG}) to $\check G$, we obtain that for any $\check\Phi$ that reduces $\Phi$, 
		 for some $\eta>0$ we have 
		 ${\check\Delta}^{\check\Phi}_{\mu}L(\check{w})<-\eta$ 
		 for all $\check{w}\in\check{\mbW}\setminus\{\varepsilon\}$. 
		Combining this with (\ref{eq:driftL}), and recalling that $\mbW\subset \check{\mbW}$ we obtain that ${\Delta}^{\Phi}_{\mu}L({w}) <-\eta$ for all $w\in\mbW\setminus\{\varepsilon\}$, 
		which concludes the proof. 
\end{proof}
\begin{proof}[Proof of Proposition \ref{prop:extppartite}]
(i) Remark that for any $\hat w\in\hat\mbW$, the set $\{i\in\maV\,:\,|\hat w|_i>0\}$ is again an independent set of $\check G$. So we can apply, for any  
$\hat{\mu}\in\textsc{Ncond}(\hat{G})$, the exact same argument as for assertion (i) in Theorem \ref{thm:ppartite}, by replacing $\maV_1$ 
by $\maV_1\cup\cop{\maV_1}$. 

(ii) 
Let $\hat\mu$ be a an element of $\mathscr M(\hat{\maV})$ whose reduced measure $\mu$ belongs to $\textsc{Ncond}(\check G)$. Let $\hat\Phi$ be an admissible policy on $\hat{\maV}$, 
$\Phi$ be a policy on $\maV$ such that $\hat\Phi$ extends $\Phi$, and $\check\Phi$ be a policy reducing $\Phi$ on $\check G$. 

First, as in (\ref{eq:majoreG}) there exists $\eta>0$ such that $\check\Delta^{\check\Phi}_{\mu} L(w)<-\eta$ for any $w\in \mbW\setminus\{\varepsilon\}$. 

Fix $\hat w$ in $\hat{\mbW}\setminus\{\varepsilon\}$. 
Then define the permutation $\gamma$ of $\hat{\maV}$ by
\[\begin{cases}
\gamma(i) &=\underline i \mbox{ and } \gamma(\underline i) =i \mbox{ if } |\hat w|_{\underline i}>0\mbox{ and } |\hat w|_{i}=0,\,i\in \maV_1,\\
\gamma(j) &=j\,\mbox{ else, for any }j\in \hat{\maV}.
\end{cases}\]
Let us also denote by $\gamma(\hat w)$, the word obtained from $w$ by replacing the letters of $\hat w$ by their image through $\gamma$, in other words for all $i\in \llbracket 1,|\hat w|\rrbracket$, 
$\gamma(\hat w)_i=\gamma(\hat w_i)$. Observe that $\gamma(\hat w)$ is clearly an element of $\mbW\setminus\{\varepsilon\}$, so in view of the above observation we have that 
\begin{equation}
\label{eq:finala}
\check\Delta^{\check\Phi}_{\mu} L(\gamma(\hat w))<-\eta.
\end{equation}
Now, as $i$ and $\underline i$ have the same connectivity in $\hat G$ for any $i\in\maV_1$, for all $n$ the conditional distribution of $\hat W_{n+1}$ given $\{\hat W_n = \hat w\}$ in the $(\hat G,\hat\Phi,\hat\mu)$ 
system equals that of  $\hat W_{n+1}$ given $\{\hat W_n = \gamma(\hat w)\}$ in the $(\hat G,\hat\Phi,\hat\mu\circ\gamma)$ system. In particular, we have that 
\begin{equation}
\label{eq:finalb}
\Delta^{\hat\Phi}_{\hat\mu}(\hat w) = \Delta^{\hat\Phi}_{\hat\mu\circ\gamma}(\gamma(\hat w)).
\end{equation}
On the other hand, as $\gamma(\hat w)$ is an element of $\mbW$ and 
the measure $\hat\mu\circ\gamma \in \mathscr M(\hat{\maV})$ clearly extends the measure $\mu$, the right inequality of (\ref{eq:driftL}) implies that
\begin{equation*}
\hat{\Delta}^{\hat\Phi}_{\hat\mu\circ\gamma}L(\gamma(\hat w)) \le \check{\Delta}^{\check\Phi}_{\mu}L(\gamma(\hat w)),
\end{equation*}
and it follows from (\ref{eq:finala}-\ref{eq:finalb}) that $\Delta^{\hat\Phi}_{\hat\mu}(\hat w)<-\eta$. As this is true for any $\hat w$ in $\hat{\mbW}\setminus\{\varepsilon\}$, the proof is complete. 
\end{proof}

\end{document}